\documentclass[opre,nonblindrev]{oo}

\OneAndAHalfSpacedXI 

\usepackage{endnotes}
\let\footnote=\endnote

\usepackage{amsmath}
\usepackage{amssymb}
\usepackage{url}
\usepackage[colorlinks = true,
            linkcolor = blue,
            urlcolor  = blue,
            citecolor = blue,
            anchorcolor = blue]{hyperref}
\usepackage{enumitem}
\usepackage[utf8]{inputenc}
\usepackage[english]{babel}
\usepackage[table,xcdraw]{xcolor}
\usepackage{booktabs}
\usepackage{comment}
\usepackage{ushort}
\usepackage{subcaption}

\newcommand{\bsubeq}{\begin{subequations}}
\newcommand{\esubeq}{\end{subequations}}
\newcommand{\BI}{\begin{itemize}}
\newcommand{\EI}{\end{itemize}}

\newcommand{\cred}{\color{red!85!black}}

\renewcommand{\mc}{\mathcal}
\newcommand{\mb}{\mathbb}

\def\st{{\rm s.t.}}

\usepackage{float}
\usepackage[ruled,vlined]{algorithm2e}
\usepackage{amssymb}
\usepackage{enumerate}
\usepackage[normalem]{ulem}

\usepackage{placeins} 

\newcommand{\T}{\mc{T}}

\newcommand{\G}{\mc{G}}

\newcommand{\Gt}{\mc{G}^{\rm{Ther}}}
\newcommand{\Gr}{\mc{G}^{\rm{Renew}}}
\newcommand{\ghat}{\hat{g}}
\newcommand{\gtilde}{\tilde{g}}
\newcommand{\gdot}{\dot{g}}
\newcommand{\gddot}{\ddot{g}}
\newcommand{\Ghat}{\hat{\mc{G}}}
\newcommand{\B}{B}
\renewcommand{\H}{H}
\newcommand{\Qsold}{Q^{\rm{Sold}}}
\newcommand{\qsoldp}{Q^{\rm{Sold}''}}
\newcommand{\wlb}{\ushort{W}_1}
\newcommand{\wub}{\bar{W}_1}
\newcommand{\Ccone}{C^{\rm{CONE}}}
\newcommand{\cvoll}{C^{\rm{VOLL}}}
\newcommand{\ptransu}{\bar{P}^{\rm{Trans}}_{ij}}
\newcommand{\ptransl}{\underline{P}^{\rm{Trans}}_{ij}}
\newcommand{\punmet}{p^{\rm{Unmet}}_{it}}
\newcommand{\punmetstar}{p^{\rm{Unmet}*}_{it}}
\newcommand{\Qcap}{Q^{\rm{CAP}}}
\newcommand{\Gset}{\mb{G}}
\newcommand{\lambdasit}{\lambda^*_{s,i(1),t}}
\newcommand{\lambdasitp}{\lambda^{*'}_{s,i(1),t}}
\newcommand{\vx}{\mathbf{x}}

\newcommand{\That}{\hat{\mc{T}}}
\newcommand{\Thatp}{\hat{\mc{T}}^{'}}

\usepackage{longtable}
\newcommand{\sm}{\setminus}
\newcommand{\gain}{Gain}
\newcommand{\loss}{Loss}
\newcommand{\lossbar}{\overline{Loss}}

\usepackage{mathtools}

\usepackage{multirow}
\usepackage{graphicx}
\usepackage{caption}
\usepackage{breqn}
\allowdisplaybreaks
\usepackage[super]{nth}

\usepackage{tikz}
\usepackage{pgfplots}
\pgfplotsset{compat=1.14}
\usetikzlibrary{shapes,arrows.meta, arrows,quotes,positioning,tikzmark,fit,backgrounds}
\tikzstyle{block} = [rectangle, draw, fill=gray!20, text width=5em, text centered, rounded corners, minimum height=4em]
\tikzstyle{block2} = [rectangle, draw, fill=gray!20, text width = 17em, text centered, rounded corners, minimum height=1em]
\tikzstyle{box} = [rectangle,text width=5em,text centered]
\tikzstyle{line} = [draw, -Latex]

\usepackage{natbib}
 \bibpunct[, ]{(}{)}{,}{a}{}{,}%
 \def\bibfont{\small}%
 \def\bibsep{\smallskipamount}%
\TheoremsNumberedThrough     
\ECRepeatTheorems

\EquationsNumberedThrough    

\usepackage{placeins}
\renewcommand{\subparagraph}{}
\usepackage{titlesec}
\titlespacing*{\section}{0pt}{5pt}{3pt}
\titlespacing*{\subsection}{0pt}{3pt}{0pt}
\titlespacing*{\subsubsection}{0pt}{3pt}{0pt}
\begin{document}


\RUNAUTHOR{Guo, Kroer, Dvorkin, Bienstock}

\RUNTITLE{Incentivize Investment and Reliability: A Study on Capacity Markets}

\TITLE{Incentivizing Investment and Reliability: A Study on Electricity Capacity Markets}

\ARTICLEAUTHORS{%
\AUTHOR{Cheng Guo}
\AFF{School of Mathematical and Statistical Sciences, Clemson University, Clemson, SC 29634, \EMAIL{cguo2@clemson.edu}} 
\AUTHOR{Christian Kroer}
\AFF{Department of Industrial Engineering and Operations Research, Columbia University, New York, NY 10027, \EMAIL{christian.kroer@columbia.edu}}
\AUTHOR{Yury Dvorkin}
\AFF{Department of Civil and System Engineering, Johns Hopkins University, Baltimore, MD 21218, \EMAIL{ydvorki1@jhu.edu}}
\AUTHOR{Daniel Bienstock}
\AFF{Department of Industrial Engineering and Operations Research, Columbia University, New York, NY 10027, \EMAIL{dano@columbia.edu}}
} 

\ABSTRACT{The capacity market, a marketplace to exchange available generation capacity for electricity production, provides a major revenue stream for generators and is adopted in several U.S. regions. A subject of ongoing debate, the capacity market is viewed by its proponents as a crucial mechanism to ensure system reliability, while critics highlight its drawbacks such as market distortion. Under a novel analytical framework, we rigorously evaluate the impact of the capacity market on generators' revenue and system reliability. More specifically, based on market designs at New York Independent System Operator (NYISO), we propose market equilibrium-based models to capture salient aspects of the capacity market and its interaction with the energy market. We also develop a leader-follower model to study market power. We show that the capacity market incentivizes the investment of generators with lower net cost of new entry. It also facilitates reliability by preventing significant physical withholding when the demand is relatively high. Nevertheless, the capacity market may not provide enough revenue for peaking plants. Moreover, it is susceptible to market power, which necessitates tailored market power mitigation measures depending on market dynamics. We provide further insights via large-scale experiments on data from NYISO markets. 

}

\SUBJECTCLASS{Energy policies; reliability; bidding/auctions.}

\AREAOFREVIEW{Environment, Energy, and Sustainability.}


\maketitle

%


\section{Introduction}
The 2021 Texas power crisis reignited debates on the energy-only market design of the Electric Reliability Council of Texas (ERCOT). ERCOT and Southwest Power Pool (SPP) are the only independent system operators (ISOs) in the US that implement energy-only markets. In all other ISOs, generators earn extra payments from some form of \emph{capacity market}, where they are paid for providing available capacity. Since generators in the energy-only market do not get revenue from the capacity market, they rely on high peak-hour prices to make adequate revenue. Those high peak-hour prices pose several challenges in the energy-only market, including price volatility and high electricity prices, such as the \$9000/MWh price during the Texas crisis. 

An important goal of both the capacity market and the energy-only market is to support system reliability via {\it resource adequacy}, i.e., maintaining a robust generation portfolio in the system. This is achieved by ensuring sufficient income for generators to cover both investment and operational costs. While the capacity market is widely adopted and is believed to offer benefits such as price stability, certain perceived drawbacks prevent its broader acceptance. Thus, there is a pressing need to evaluate the capacity market's performance, which is a topic that has not been very well studied \citep{gao_report_2017}. In this work, we focus on evaluating the capacity market's impact on generators' revenue and system reliability. 

Formally, the capacity market is a forward market that trades available electricity generation capacity at a future time. Before the introduction of capacity markets, existing market mechanisms failed to incentivize adequate investment into new generation capacity, which in turn undermines system reliability. Payments from the capacity market can be seen as extra revenue streams to generators, complementary to the revenue from selling electricity in the {\it energy market}, which is the market for trading electricity. \cite{cramton2013capacity} show that the capacity market ensures resource adequacy, reduces physical withholding, and decreases price volatility. \cite{nrel_report_2016} point out that as growing levels of renewable energy suppresses electricity prices, the capacity market becomes an increasingly important source of income for generators. On the other hand, \cite{oren2000capacity} mentions that the capacity market could distort energy prices and result in over-investment, and \cite{mays2019asymmetric} find that the current design of the capacity market may favor high-carbon resources. In addition, \cite{nyiso_mitigation_2019} shows that in NYISO suppliers in the capacity market may submit uneconomic bids that suppress the market clearing prices. In this work, we assess the capacity market under a rigorous analytical framework, featuring novel market equilibrium-based models based on NYISO practice, as well as a new perspective that considers the interplay between the capacity market and the energy market. Our analysis on the capacity market focuses on its impacts on generators' revenue and the subsequent implications for market outcomes, and we strive to provide rich insights for both regulators and market participants.

More broadly, studying the capacity market could shed light on an important question in mechanism design: how to provide incentives for investment. When selling electricity in the energy market, generators set bidding prices based on marginal costs, while as pointed out by \cite{budde1997capacity}, such bidding prices may not provide incentive for investment. Recently, \cite{akbarpour2022algorithmic} distinguish between investment and allocative guarantees of a mechanism, outlining conditions under which those guarantees coincide. In this regard, the joint capacity and energy markets, prevalent in most ISOs in the U.S., can be viewed as a mechanism designed aiming to achieve this objective: supporting optimal decisions in both investment and allocation. By examining the interplay between the capacity and energy markets and its impact on generators' revenue, we offer a deeper understanding into the implications of this real-life market design.

\subsection{Main Contributions}
The capacity market is an important component of power system operations, yet its performance is not very well studied \citep{gao_report_2017}, nor is its effectiveness in ensuring resource adequacy fully assessed \citep{cap_report_2018}. We propose tractable equilibrium-based and game-theoretic models rooted in real-life practice, both analytically and numerically examine the outcome of the capacity market and its implications on generators' revenue and system reliability.

 It is challenging to conduct economic analysis on large-scale capacity and energy markets, given the complicating physical constraints and market features. One contribution of our work is to develop a novel analytical framework, which enables economic analysis on realistic models of capacity and energy markets. To date, the majority of large-scale studies on the capacity market are based on computational simulations, while analytical studies employing stylized models often make simplifying assumptions regarding market features, such as the market size and the clearing process. Our method enables a more accurate and detailed assessment of the capacity market, offering nuanced discussions and novel insights. In addition, we are able to connect our theoretical results with several observations in real-life practice and in our case study.  

Furthermore, our work has several innovations in terms of modeling and computation. Firstly, we propose a novel model which captures salient aspects of the NYISO capacity market auction. We model the market clearing process with a quadratic convex (QC) optimization problem, which is tractable and can potentially be used for other uniform-price auctions with a similar structure. Secondly, thanks to the convexity of the QC model, we build a trilevel model for a leader-follower game in joint capacity and energy markets. To speed up the solution of the trilevel model, we propose a valid inequality and reformulation techniques by leveraging our insights into the capacity market, which make it possible to solve a large-scale NYISO-based case study. Finally, the detailed modeling of the capacity market enables us to assume a new perspective in this study, where we consider the interplay between the capacity and energy markets, as well as its impact on generators' revenue, which leads to novel insights.

In terms of our findings, we identify a key factor that impacts the effectiveness of the capacity market in ensuring resource adequacy: the net cost of new entry (CONE) of generators. The net CONE is the magnitude of generators' deficit after they receive payments in the energy market. We find that the capacity market is positioned to ensure profitability of generators with lower net CONE, but may not sufficiently reimburse generators with high net CONE. Also, generators with high net CONE are more likely to benefit from exercising market power in the capacity market.

Another factor that influences the performance of the capacity market is the demand level. The capacity market enjoys better properties at higher demand levels, as it makes more generators profitable, is less likely to be affected by market power, and is more effective at reducing physical withholding in the energy market. On the other hand, when demand is low or if the market is sparse, market power could become a more prominent problem. We suggest different market mitigation measures based on the demand level, demand elasticity, and market density.

Our analysis also deepens our understanding on how the capacity market impacts system reliability. The capacity market enhances reliability by maintaining generators with lower net CONE within the system, by stabilizing electricity price, and by preventing substantial physical withholding in a system with limited backup capacity. However, the capacity market may not incentivize investment in peaking plants, which are essential for system reliability during periods of peak demand. In addition, its efficacy in mitigating physical withholding diminishes in the presence of congestion or load shedding. We conclude that the capacity market can be a valuable instrument in improving system reliability, while additional measures to incentivize peaking plants investment and to alleviate congestion and load shedding issues would further contribute to this objective.

\subsection{Literature Review}\label{ch: lit_review}


There have been a variety of optimization-based methods for modeling and studying capacity payments. For example, \cite{ehrenmann2011generation} use the duals of investment constraints to represent capacity payments, which differs from the current practice where the payments are based on revenue deficit. \cite{levin2015electricity} incorporate capacity payments in their model as fixed values. \cite{mays2019asymmetric} model the capacity payments with financial products such as call options and futures. \cite{kwon2018stochastic} propose an optimization model for capacity market clearing. 
While our modeling setup is also based on market clearing, we strive to provide a detailed model for the spot market auction in the NYISO capacity market, including features such as a linear demand curve and net CONE-based offer prices, which enables more accurate and thorough evaluation of the capacity market. In addition, the majority of existing literature focuses on numerical case studies, while we undertake rigorous analytical study to provide novel insights. 

An important objective of the capacity market is to ensure sufficient income for generators, yet the literature provides divided conclusions
on whether the capacity market is necessary for this purpose. \cite{oren2000capacity} argues that if spot prices in the energy market correctly reflect scarcity rents, they will be enough to provide sufficient income for generators and thus capacity payments are not needed. In addition, even if spot prices are not able to incentivize adequate investment, there still exist more preferable methods than administratively set capacity payments. On the other hand, \cite{joskow2008capacity} uses a two-state model to show that the capacity market helps to resolve the ``missing money" problem and thus supports generator profitability. A case study on PJM \citep{sioshansi2013evolution} finds that capacity market revenue in PJM provides enough incentive for both coal and natural gas generators to remain in the market, but does not provide the incentive for new coal generators to enter. In this study, utilizing our proposed models for analysis, we provide fresh insights, such as how net CONE and demand levels affect the revenue and behavior of generators.

While there is a rich literature studying market power in the energy market \citep{li2011modeling}, market power in the capacity market is relatively not well-studied. \cite{schwenen2015strategic} investigates the strategic bidding behavior in the NYISO capacity market. A model of multi-unit uniform price auctions is used to explain observed bidding behavior in NYISO. Their work focuses on a setup where all generators collude to maximize their profit. In contrast, our work uses a leader-follower game framework and considers additional market features.


The rest of this paper is structured as follows. Section \ref{ch: model} describes modeling setup and proves some properties of the capacity market. Section \ref{ch: combined_market} explores the interplay between capacity and energy markets, providing further insights on the influence of the capacity market. Section \ref{ch: stackelberg} studies market outcomes in the presence of market power. Section \ref{ch: experiment} provides an NYISO-based case study. 

All proofs are provided in Section \ref{ec: proofs} of the e-companion.
\section{Models and Properties for the Capacity Market}\label{ch: model}
In the capacity market, load-serving entities (i.e., the buyers), known as LSEs, procure capacity to meet their capacity requirements, while the generators (i.e., the sellers), which sell their capacity, are required to make the sold capacity is available in the energy market. For this work, we design our model based on the implementation at NYISO \citep{nyiso_icap_2021}. In NYISO, there are 3 types of capacity auctions: 6-month auction, monthly auction, and spot market auction. To streamline our models and focus on salient aspects of the capacity market, we only consider the spot market auction, which is held 4-5 days prior to the start of the upcoming month, and assume that all capacity is traded in this auction. For the remainder of the paper, we use the terms ``spot market" and ``capacity market" interchangeably. 

In this section and in Section \ref{ch: combined_market}, we consider market participants as price takers without market power, with perfect information, and without externality. Thus, the capacity market satisfies the conditions of perfect competition. Later, in Section \ref{ch: stackelberg} we consider the impact of market power.

\subsection{Market Clearing in Capacity Market Auction}\label{ch: cm_background}
In this section, we describe the NYISO spot market auction and its market clearing process in detail. We also present some findings and discussions regarding the market outcome of the spot market auction.

Before delving into the details of the market setup, we introduce a key concept in the capacity market: the net CONE (depending on its unit, net CONE can also be called ``annual reference value" or ``reference point price"). The net CONE is defined as (the positive portion of) the {\it total investment and operational costs} minus the {\it energy and ancillary service (E\&AS) revenues}. More specifically, the total investment and operational costs include the investment cost, fixed operations and management (O\&M) expenditures, and the variable costs of generation; the E\&AS revenues include revenue from the energy market and the ancillary services market. To streamline our model, we do not consider the ancillary services market or the fixed O\&M expenditures. Consequently, the net CONE can be calculated as follows:
\begin{align}\label{eq: def_cone}
\text{net CONE = (investment cost - energy market profit)$^+$,}
\end{align}
where the energy market profit equals energy market revenue minus variable costs of generation, and $(\cdot)^+$ represents $\max(0, \cdot)$. For this work, all values in \eqref{eq: def_cone} are in units of \$/MW-day.

We denote the net CONE of generator $g$ as $W_g$. For a peaking plant, its net CONE is also denoted as $\Ccone$. Note that a peaking plant, also called a peaker, is a power plant that usually only operates when there is a high demand. It is often a power plant that can rapidly start up and it usually has the highest variable costs.  

For the demand side of the spot market auction, the NYISO submits monthly bids on behalf of all LSEs in the form of a demand curve, as demonstrated by the pink curve in Figure \ref{fig: origin_market_eq}. This demand curve is piecewise-linear, with a flat part and a downward linear part. It is decided by the following 3 points: 

- $(0, P_1)$ is called the maximum clearing point, where $P_1$ is the maximum market clearing price and it equals 1.5 times the levelized cost (i.e., average lifetime cost per unit) of a peaker. A simple levelized cost formula is provided in \cite{lcoe}. In our model, the lifetime cost equals the total investment and variables costs. 

- $(\Qcap, \Ccone)$ is called the reference point, where $\Qcap$ is the total capacity requirement of all LSEs. Let $D^{Peak}$ be the peak load (i.e., highest demand) of the time horizon under consideration, $F^{TF}$ be the translation factor, and $\Gamma$ be the installed reserve margin. Then $\Qcap = (1 - F^{TF})(1+\Gamma)D^{Peak}$. To understand the formulation for $\Qcap$, note that $(1+\Gamma)D^{Peak}$ is the installed capacity (ICAP) requirement, which is set at a value that is higher than the peak load. This value is then multiplied by $1 - F^{TF}$ to obtain the unforced capacity requirement, which is the capacity that generators should relialy supply after accounting for the probability of them being unavailable. We use $\Gamma = 20.70\%$ and $F^{TF} = 8.56\%$ in the case study.

- $(Q_3, 0)$ is called the zero-crossing point, where $Q_3$ is the minimum quantity beyond which additional capacity has 0 value. This quantity is higher than $\Qcap$ by a percentage $F^E$, i.e. $Q_3 = (1 + F^E) \Qcap$. We use $F^E = 18\%$ in the case study.

To draw the demand curve, we draw a line through $(\Qcap, \Ccone)$ and $(Q_3, 0)$, and this line intersects with the horizontal line that extends from $(0, P_1)$. Let the expression for the curve through $(\Qcap, \Ccone)$ and $(Q_3, 0)$ be $P = - AQ + \Pi^\max$, then we have $A = \frac{\Ccone}{F^E \Qcap}$ and $\Pi^\max = \frac{1+ F^E}{F^E}\Ccone$.

\begin{figure}[htbp]
\centering
\begin{subfigure}[t]{0.5\textwidth}
\centering
\begin{tikzpicture}[scale=0.55]
    \begin{axis}[
        axis lines=left,
        xlabel={$Q$},
        ylabel={$P$},
        xmin=0, xmax=10,
        ymin=0, ymax=10.7,
        xtick=\empty,
        ytick=\empty,
        xticklabels={}, yticklabels={},
        minor tick num=1,
        extra x ticks={4.8, 5.5, 8},
        extra y ticks={3, 4, 6},
        extra x tick labels={$\Qcap~~~$, $~~~r^*$, $Q_3$},
        extra y tick labels={$\pi^*$, $\Ccone$, $P_1$},
        tick label style={font=\Large}, 
        xlabel style={at={(xticklabel* cs:1)}, anchor=west, font=\Large}, 
        ylabel style={at={(yticklabel* cs:1)}, anchor=east, rotate = 270, font=\Large}, 
        axis line style={line width=1pt},
    ]
    
    \addplot[magenta, domain=0:3, line width=1.5pt] {6};
            
    \addplot[magenta, domain=3:8, line width=1.5pt] {-6/5 *x + 48/5};


    \addplot[violet, domain=0:2, line width=1.5pt] {1};
    \addplot[violet, domain=2:4, line width=1.5pt] {2};
    \addplot[violet, domain=4:7, line width=1.5pt] {3};
    \addplot[violet, domain=7:9, line width=1.5pt] {4};

    \draw[dotted, line width=1.5pt] (4.8,0) -- (4.8,4);
    \draw[dotted, line width=1.5pt] (0,4) -- (4.8,4);
    
    \draw[dotted, line width=1.5pt] (5.5,0) -- (5.5,3);
    \draw[dotted, line width=1.5pt] (0,3) -- (5.5,3);
    \filldraw[black, line width=1.5pt] (5.5,3) circle (2pt);
    
    \end{axis}
\end{tikzpicture}
\caption{}\label{fig: origin_market_eq}
\end{subfigure}
\hfill
\begin{subfigure}[t]{0.5\textwidth}
\centering
\begin{tikzpicture}[scale=0.55]
    \begin{axis}[
        axis lines=left,
        xlabel={$Q$},
        ylabel={$P$},
        xmin=0, xmax=10,
        ymin=0, ymax=10.7,
        xtick=\empty,
        ytick=\empty,
        xticklabels={}, yticklabels={},
        minor tick num=1,
        extra x ticks={5.5, 8},
        extra y ticks={3, 9.6},
        extra x tick labels={$r^*$, $Q_3$},
        extra y tick labels={$\pi^*$, $\Pi^{\max}$},
        tick label style={font=\Large}, 
        xlabel style={at={(xticklabel* cs:1)}, anchor=west, font=\Large}, 
        ylabel style={at={(yticklabel* cs:1)}, anchor=east, rotate = 270, font=\Large}, 
        axis line style={line width=1pt},
    ]
    
            
    \addplot[magenta, domain=0:8, line width=1.5pt] {-6/5 *x + 48/5};

    \addplot[domain=0:5.5, fill=cyan, fill opacity=0.1, draw = none] {-6/5 *x + 48/5} \closedcycle;
    \addplot[domain=0:5.5, fill=white, draw = none] {3} \closedcycle;

    \addplot[domain=0:4, fill=red, fill opacity=0.1, draw = none] {3} \closedcycle;
    \addplot[domain=0:2, fill=white, draw = none] {1} \closedcycle;
    \addplot[domain=2:4, fill=white, draw = none] {2} \closedcycle;

    \addplot[violet, domain=0:2, line width=1.5pt] {1};
    \addplot[violet, domain=2:4, line width=1.5pt] {2};
    \addplot[violet, domain=4:7, line width=1.5pt] {3};
    \addplot[violet, domain=7:9, line width=1.5pt] {4};

    
    \draw[dotted, line width=1.5pt] (5.5,0) -- (5.5,3);
    \draw[dotted, line width=1.5pt] (0,3) -- (5.5,3);
    \filldraw[black, line width=1.5pt] (5.5,3) circle (2pt);
    
    \end{axis}
\end{tikzpicture}
\caption{}\label{fig: updated_market_eq}
\end{subfigure}
\caption{\small Illustrations of market equilibrium with: (a) original demand curve, (b) linear demand curve.}
\label{fig: market_curve}
\end{figure}
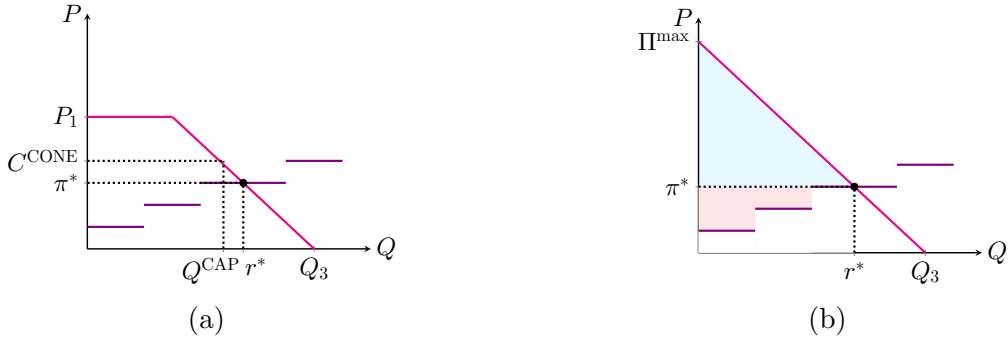



For the supply side of the spot market auction, the suppliers submit both offer price and offer capacity, and thus the supply function is a step function, as illustrated by the purple curve in Figure \ref{fig: origin_market_eq}. When a generator has no market power, it bids truthfully with an offer price equal to its net CONE $W_g$, which represents the long-run marginal cost of capacity, and an offer capacity equals its qualified capacity (which will be defined soon). The intersection of supply and demand curves in Figure \ref{fig: origin_market_eq}  represents the market equilibrium, denoted as $(r^*, \pi^*)$, where $r$ and $\pi$ are respectively variables for cleared capacity and market clearing price in the capacity market.


\begin{remark}\label{rm: cm_clearing}
One might be wondering whether the LSEs are able to procure enough capacity in the auction to meet their capacity requirement. This is equivalent to asking whether $r^* \geq \Qcap$, or equivalently, whether $\pi^* \leq \Ccone$. Those conditions are satisfied when there is enough capacity with an offer price that is no greater than $\Ccone$. Note that as shown in \cite{isone_cm} and also in our case study, the net CONE of the peaker is the highest among all generators (i.e., $\Ccone = \max_{g\in\mc{G}} W_g$). Therefore, as long as the market clears, the market clearing price $\pi^*$ cannot be more than $\Ccone$, and thus $r^*\geq \Qcap$.
\end{remark}

Remark \ref{rm: cm_clearing} indicates that the market equilibrium point cannot lie on the flat piece of the demand curve, and consequently we can replace the piecewise-linear demand curve  with a linear demand curve through $(\Qcap, \Ccone)$ and $(Q_3, 0)$. This approximation simplifies our model and does not affect the market outcome. We show the market equilibrium with an updated demand curve in Figure \ref{fig: updated_market_eq}. The {\it demand function} is now $P = - AQ + \Pi^\max, 0 \leq Q \leq Q_3$.

To describe the market clearing process, we need some additional notation. For generator $g$, let $P^\max_g$ be its capacity and $F^U_g$ be its unforced capacity percentage. Then its qualified capacity for sale equals $F^U_g P^\max_g$. We consider both thermal and wind generators. For thermal generators, $F^U_g=1$ as they provide ``firm" capacity without too much uncertainty; For wind generators, $F^U_g < 1$ as their production levels are less predictable (e.g., in our case study we use $F^U_g = 0.24$ for wind generators). To ensure that there is enough capacity to fulfill the capacity requirement, we assume $\sum_{g\in\mc{G}} F^U_g P^\max_g \geq \Qcap$, i.e. the total qualified capacity is no less than the capacity requirement. Moreover, for generator $g$ let $h_g$ be its offer capacity and $q_g$ be its capacity sold in the capacity market, and we have $0\leq q_g \leq h_g\leq F_g^U P^\max_g$. 

We can use a greedy algorithm to identify generators that clear the market. First, we sort the generators according to their offer price $W_g$ in an ascending order. When the capacity market clears, the generators with the lowest offer prices are selected. Since we have a linear demand function $P = - AQ + \Pi^\max$, if the market clears at $\pi^* = W_{\ghat}$ with $\ghat$ being the marginal supplier that clears the market, then the total cleared quantity $r^* = \frac{\Pi^\max - W_{\ghat}}{A}$. To find $\ghat$, the algorithm iteratively checks if the total offer capacity up to generator $g$ exceeds $\frac{\Pi^\max - W_g}{A}$, and stops when it finds the first generator that does not. This generator is the marginal supplier. The sold capacity $q_g$ equals $h_g$ if $W_g<W_{\ghat}$, and equals 0 if $W_g>W_{\ghat}$. For the marginal supplier, $q_{\ghat} = r^*-\sum_{g=1}^{\ghat-1} h_g$. On the other hand, if the algorithm fails to find a marginal generator, then the market does not clear. The pseudocode of this algorithm is provided in Section \ref{ch: greedy} of the e-companion. This procedure has a complexity of $O(|\mc{G}|\log(|\mc{G}|))$ due to sorting, with $\mc{G}$ being the set of generators. Note that in our analysis, we assume all net CONEs are different so there is no tie to break. This assumption is later relaxed in the case study to provide a more nuanced discussion.

With the market clearing process explained, we now formalize and advance the results of Remark \ref{rm: cm_clearing} in the following proposition, which provides an expression for the proportion of excess capacity and shows that it is nonnegative. The excess capacity is 0 only if the peaker clears the market:

\begin{proposition}\label{th: excess_cap}
If the capacity market clears, then the proportion that the market clearing quantity exceeds the capacity requirement $\frac{r^* - \Qcap}{\Qcap} = (1 - \frac{W_{\ghat}}{\Ccone})F^E \geq 0$. In particular, if the marginal supplier $\ghat$ is the peaker, then $r^* = \Qcap$.
\end{proposition}


In practice, if the ISO procures excess capacity, then the excess capacity will be allocated to each LSE proportionally according to its load-ratio share (i.e., the LSE's portion of load compared with the total load).

Although the greedy algorithm is straightforward to implement for obtaining capacity market outcomes, its utility can be limited when used to analyze those outcomes. Therefore, we propose a novel QC optimization model for the capacity market in the following section. We additionally provide a novel mixed-integer programming (MIP) model in Section \ref{ch: mip} of the e-companion, which is inspired by the greedy algorithm and offers an alternative perspective to facilitate a deeper understanding of the market clearing process. Both of those models can be used in other similarly-structured markets with a linear demand function and a step supply function. 


\subsection{Quadratic Convex Model for Capacity Market}\label{ch: qc}
As previously mentioned, in the absence of market power, the capacity market satisfies the conditions of perfect competition. As a result, the market equilibrium point maximizes social welfare. Essentially, our QC model is designed to find the social welfare-maximizing price and quantities.

The social welfare equals the sum of producer and consumer surpluses. In Figure \ref{fig: updated_market_eq}, the producer surplus is shaded in red and equals the total profit of generators, which is expressed as $\sum_{g\in\mc{G}} (\pi - W_g) q_g$; the consumer surplus is shaded in blue and equals the area of the triangle which is below the demand curve and above the line for the market clearing price, and it can be calculated by $\frac{\Pi^\max - \pi}{2} r$. The social welfare thus equals $\sum_{g\in\mc{G}} (\pi - W_g) q_g + \frac{\Pi^\max - \pi}{2} r$. 

This expression for social welfare is nonconvex because of the bilinear terms $\pi q_g$ and $\pi r$. To use it in our QC model, we need to transform the expression to a convex quadratic form. To do this, we substitute $\pi$ using the demand function $\pi = - A \sum_{g\in\mc{G}} q_g + \Pi^\max$, and substitute $\sum_{g\in\mc{G}} q_g$ using the relationship $\sum_{g\in\mc{G}} q_g=r$. The updated expression for social welfare is $- \frac{A}{2} r^2 + \sum_{g\in\mc{G}} (\Pi^\max - W_g) q_g$.

We propose the following QC model for the capacity market:
\begin{subequations}\label{eq: cm_qc}
\begin{alignat}{4}
\text{(CM-QC):}\max~~& - \frac{A}{2} r^2 + \sum_{g\in\mc{G}} (\Pi^\max - W_g) q_g\label{eq: cm_qc_0}\\
\st~~& r = \sum_{g\in\mc{G}} q_g\label{eq: cm_qc_1}\\
& h_g \leq F^U_g P_g^\max&&  \forall g\in\mc{G}  \hspace{1cm}&&\label{eq: cm_qc_2}\\
& q_g\leq  h_g && \forall g\in\mc{G}&&\label{eq: cm_qc_3}\\
& h_g, q_g\geq 0 && \forall g\in\mc{G},\label{eq: cm_qc_4}
\end{alignat}
\end{subequations}
where the objective \eqref{eq: cm_qc_0} maximizes the social welfare. \eqref{eq: cm_qc_1} defines the cleared capacity $r$. \eqref{eq: cm_qc_2} and \eqref{eq: cm_qc_3} respectively set the upper bounds for the offer capacity and the sold capacity.

This is a convex optimization problem because we have a maximization problem and the coefficient for the quadratic term in the objective \eqref{eq: cm_qc_0} is negative. In addition, all constraints are linear. Since (CM-QC) is convex, we are able to use its KKT conditions to represent the market equilibrium of the capacity market in Section \ref{ch: stackelberg}.

With this model, we can prove that social welfare is maximized when each generator offers its entire qualified capacity into the capacity market:

\begin{proposition}\label{th: cm_h_opt}
There exists an optimal solution for the QC model \eqref{eq: cm_qc} with $h_g = F^U_g P^\max_g, \forall g\in\mc{G}$.
\end{proposition}

It is possible that there are multiple optimal solutions for $h_g$. For example, when the optimal solution for $q_g$ is 0, $h_g$ can take any value in $[0, F^U_g P^\max_g]$ without affecting the optimal value. 

Although Proposition \ref{th: cm_h_opt} relies on the QC model, the conclusion remains true even if the market equilibrium is obtained with other methods, as all methods lead to the same optimal solution(s). 

\subsection{Capacity Market and Revenue Adequacy}\label{ch: cm_profit}

An important objective of the capacity market is to ensure generators make sufficient income. In this section, we analyze how the capacity market impacts the profit of generators.

The profit of a generator can be assessed by comparing its capacity market revenue with its net CONE, without considering the outcome of the energy market. This is because when calculating the net CONE, we have already deducted energy market profit from its value. Also, as we will discuss shortly in Section \ref{ch: em_model}, when there is no market power, the outcome of capacity market does not affect the revenue in energy market. Therefore, as long as the net CONE is correctly estimated, it should reflect how much additional revenue the generator needs from the capacity market to become profitable. For example, if a generator's net CONE equals 0, then it is already profitable and does not need more revenue. 


The following theorem provides expressions for the profits of different types of generators. It shows that the profitability of a generator depends on how it gets allocated in the capacity market:
\begin{theorem}\label{th: profit_cm}
Let $\Ghat$ be the set of generators that are allocated in the capacity market, and let $\ghat$ be the marginal supplier. Assume that the net CONE is correctly estimated, and $W_{\ghat} \neq 0$. Then the profit of generator $g$ is as follows:

(1) If $g\in\Ghat\sm\{\ghat\}$, then its profit equals $(W_{\ghat} - W_g)F^U_g P^\max_g > 0$.

(2) If $g$ is the marginal supplier $\ghat$, then its profit equals $W_{\ghat}(\frac{\Pi^\max - W_\ghat}{A} - \sum_{g\in\Ghat}F^U_g P^\max_g)\leq 0$

(3) If $g\notin\Ghat$, then its profit equals $-W_g F^U_g P^\max_g < 0$
\end{theorem}



Therefore, the capacity market does not guarantee the profitability of all generators. This result also implies that the capacity market incentivizes the investment of technologies with smaller net CONE, as those generators tend to get allocated in the capacity market. In our case study, wind, natural gas, and hydro are 3 types of generators with the smallest net CONE. 

\begin{remark}
    Since we assume $W_\ghat \neq 0$, the marginal supplier is only profitable when $\frac{\Pi^\max - W_\ghat}{A} - \sum_{g\in\Ghat}F^U_g P^\max_g = 0 \Rightarrow W_\ghat = -A\sum_{g\in\Ghat} F^U_g P^\max_g + \Pi^\max$, and is not profitable when $W_\ghat$ is higher. In other words, the marginal supplier is only profitable when its net CONE is low enough such that it is fully allocated.
\end{remark}

Since the peaker has the highest net CONE, Theorem \ref{th: profit_cm} also indicates that the peaker is unlikely to be profitable:
\begin{corollary}\label{th: peak_rev}
The peaker is revenue balanced (i.e., its revenue equals cost) only if it is the marginal generator and $\Qcap = \sum_{g\in\mc{G}}F^U_g P^\max_g$. Otherwise, it always operates at a loss.
\end{corollary}


In Table \ref{tab: cm_profit} we summarize how the capacity market affects the profitability of generators, where we list the conditions for generators to be profitable/not profitable, depending on whether it gets revenue from the capacity market. We highlight the profitability of the peaker, as its availability is essential for a reliable power grid. 

\begin{table}[htbp]
\centering
\caption{Effect of the Capacity Market (CM) on the Profitability of Generators}
\label{tab: cm_profit}
{\small
\begin{tabular}{l|ll}
\hline
      & \multicolumn{1}{c}{Profitable} & \multicolumn{1}{c}{Not profitable} \\ \hline
No CM & $W_g = 0$                       & $W_g > 0$                \\
      &                                & peaker                             \\
CM    & $g\in \Ghat \sm \{\ghat\}$                          & $g\notin \Ghat$                              \\
      & $\ghat$ \& $W_\ghat = -A\left(\sum_{g\in\Ghat} F^U P^\max_g\right) + \Pi^\max$                          & $\ghat$ \& $W_\ghat > -A\left(\sum_{g\in\Ghat} F^U P^\max_g\right) + \Pi^\max$                           \\
      & peaker $= \ghat$ \& $\Qcap = \sum_{g\in\G} F^U_g P^\max_g$                          & peaker $= \ghat$ \& $\Qcap < \sum_{g\in\G} F^U_g P^\max_g$; peaker $\neq \ghat$     \\ \hline
\end{tabular}}
\end{table}
\section{Joint Capacity and Energy Markets}\label{ch: combined_market}
The decisions in capacity market and energy market are interdependent. On one hand, procuring capacity in the capacity auction ensures its availability in the energy market. On the other hand, generator's profit in energy market affects its net CONE and consequently influences the outcome of the capacity market. In this section, our goal is to model and understand the interaction between capacity and energy markets, which leads to deeper insights into the impact of capacity market on generator profitability and overall market operation.




\subsection{A Modified Energy Market Model}\label{ch: em_model}
We model the energy market using a modified linearized direct-current optimal power flow (DCOPF) problem. Unlike the standard linearized DCOPF models, in our model generators do not need to provide their whole capacity in the energy market, which better reflects the bidding behaviors in capacity and energy markets.
More specifically, the allocated capacity in the capacity market $q_g$ is taken as given in the energy market and is denoted as $\bar{q}_g$. In addition to $\bar{q}_g$, a generator can provide $v_{g}$ capacity in the energy market. The energy market model is as follows:
\begin{subequations}\label{eq: em}
\begin{alignat}{4}
\text{(EM):}\min~~& \sum_{t\in\mc{T}}\left(\sum_{g\in\Gt} C^V_g p_{gt}+ \sum_{i\in\mc{N}}\cvoll\punmet\right)\label{eq: em_0}\\
\st~~&\sum_{g\in\mc{G}_i} p_{gt} + \punmet + \sum_{(j, i)\in\mc{E}}f_{jit} - \sum_{(i,j)\in\mc{E}} f_{ijt} = D_{it} &~~&\forall i\in\mc{N}, t\in\mc{T} &&\hspace{-1cm}(\lambda_{it})\label{eq: em_1}\\
& f_{ijt} = B_{ij}(\theta_{it} - \theta_{jt}) && \forall (i, j)\in\mc{E}, t\in\mc{T} && \hspace{-1cm}(\iota_{ijt})\label{eq: em_2}\\
& f_{ijt} \leq \ptransu &~~&\forall (i, j)\in\mc{E}, t\in\mc{T} &&\hspace{-1cm}(- \zeta_{ijt})\label{eq: em_3}\\
& f_{ijt} \geq \ptransl &~~&\forall (i, j)\in\mc{E}, t\in\mc{T} && \hspace{-1cm}(\eta_{ijt})\label{eq: em_4}\\
& p_{gt} \leq \bar{q}_g + v_{g} &~~&\forall g\in\Gt, t\in\mc{T} && \hspace{-1cm}(- \alpha_{gt})\label{eq: em_5}\\
& p_{gt} = F^{CF}_{gt}(\bar{q}_g + v_{g})  &~~&\forall g\in\Gr, t\in\mc{T} \hspace{1cm}&~~&\hspace{-1cm} (\rho_{gt})\label{eq: em_6}\\
& \bar{q}_g + v_{g}\leq P_g^\max &~~&\forall g\in\G &&\hspace{-1cm} (-\phi_{g})\label{eq: em_7}\\
& p_{gt} \geq 0 &~~& \forall g\in\Gt, t\in\mc{T}\label{eq: em_8}\\
& \punmet\geq 0 &&\forall i\in\mc{N}, t\in\mc{T}\label{eq: em_9}\\
& v_{g} \geq 0 &~~& \forall g\in\Gt.\label{eq: em_10}
\end{alignat}
\end{subequations}
The Greek letters at the end of each set of constraints are corresponding dual variables, with $\zeta_{ijt}, \eta_{ijt}, \alpha_{gt}, \phi_g\geq 0$ and the other variables free. The objective \eqref{eq: em_0} minimizes the total cost, which includes the variable production cost $C^V_g$ and the value of lost loads (VOLL) $\cvoll$. In addition, $\T$ is the set of time periods, $\Gt$ is the set of thermal generators, and $\mc{N}$ is the set of buses (nodes). $p_{gt}$ and $\punmet$ are respectively production level and unsatisfied load. Note that $\cvoll$ is set to be very high (e.g., \$1,000/MW in our model) to limit the chances of supply interruption. \eqref{eq: em_1} are power flow constraints where $f_{ijt}$ is the power flow, $\G_i$ is the set of generators at bus $i$, and $\mc{E}$ is the set of transmission lines. Those constraints ensure that for each hour the load $D_{it}$ at each bus is satisfied. \eqref{eq: em_2} are the Ohm's law constraints that connect flow with voltage angle $\theta_{it}$, where $B_{ij}$ denotes the susceptance. \eqref{eq: em_3} and \eqref{eq: em_4} enforce upper limit $\ptransu$ and lower limit $\ptransl$ on the transmission lines. \eqref{eq: em_5} ensure that the thermal generators do not produce more than their available capacity. \eqref{eq: em_6} set the production levels for renewable generators, where $F^{CF}_{gt}$ is the capacity factor and $\Gr$ is the set of renewable generators. \eqref{eq: em_7} impose bounds on the available capacity of generators. 

Similar to the case with the capacity market, we can prove that it is always optimal to have $v_{g} = P^\max_g - \bar{q}_g$ in the energy market, as shown in Proposition \ref{th: em_h_opt}:
\begin{proposition}\label{th: em_h_opt}
There exists an optimal solution for \eqref{eq: em} with $v_{g} = P^\max_g - \bar{q}_g, \forall g\in\G$.
\end{proposition}

\begin{remark}\label{rm: cm_em_ind}
Thanks to Proposition \ref{th: em_h_opt}, we can decouple the capacity market problem and the energy market problem. This is because the term $\bar{q}_g + v_{g}$ in \eqref{eq: em_5} and \eqref{eq: em_6} can be replaced with $P^\max_g$, eliminating $\bar{q}_g$ from (EM), and thus (EM) is free of any capacity market decisions.
\end{remark}

 Therefore, the outcome of the energy market is not affected by the capacity market. As a result, the capacity market can be viewed as an extra revenue stream for generators, which does not affect the revenue in the energy market. Note that this result may not be true if there exists market power, as we show in Section \ref{ch: stackelberg}.

\subsection{Energy Market Equilibrium and Net CONE} \label{ch: em_cone}
In the energy market, the demand is modelled as inelastic and fixed. Thanks to these properties, a classical result states that the primal and dual solutions of DCOPF constitute a competitive equilibrium. For our particular DCOPF model, we include the formal statement and proof for a competitive equilibrium result in energy market as Proposition \ref{th: em_eq}. Let $p_{gt}^*$, $\punmetstar$, and $\lambda_{it}^*$ respectively be the optimal solutions for $p_{gt}$, $\punmet$, and $\lambda_{it}$ in (EM). Also, we use the notation $i(g)$ for the bus where generator $g$ is located, which is introduced in the proof and used subsequently throughout the paper. We can show that:
\begin{proposition}\label{th: em_eq}
The optimal prices $(\lambda_{it}^*)_{i\in\mc{N}}$, production levels $(p_{gt}^*)_{g\in\mc{G}}$, load shedding levels $(\punmetstar)_{i\in\mc{N}}$, and the demand $(D_{it})_{i\in\mc{N}}$ form a competitive equilibrium in the energy market at time $t\in\mc{T}$.
\end{proposition}

The energy market model (EM) is a linear program with linear objective and convex feasible region. Thus, the market equilibrium of the energy market can be represented with KKT conditions for (EM). In addition to the primal constraints \eqref{eq: em_1} - \eqref{eq: em_10}, the KKT conditions also contain dual constraints and complementary slackness constraints. We list some of those constraints below which are useful in the development of our results:
\begin{subequations}\label{eq: em_kkt}
\begin{alignat}{4}
& \lambda_{i(g), t} -\alpha_{gt}  \leq C^V_g \hspace{1cm}&&\forall g\in\Gt, t\in\mc{T}\hspace{1cm}&&(p_{gt})\label{eq: em_kkt_1}\\
& \lambda_{it}\leq \cvoll&&\forall i\in\mc{N}, t\in\mc{T}&&(\punmet)\label{eq: em_kkt_3}\\
& (\lambda_{i(g),t} - \alpha_{gt} - C^V_g)p_{gt} = 0 \hspace{1cm}&&\forall g\in\Gt, t\in\mc{T}\label{eq: em_kkt_12}\\
& (\lambda_{it} - \cvoll)\punmet = 0&&\forall i\in\mc{N}, t\in\mc{T} \label{eq: em_kkt_13}\\
& (p_{gt} - \bar{q}_g - v_{g})\alpha_{gt} = 0 &&\forall g\in\Gt, t\in\mc{T} \label{eq: em_kkt_18},
\end{alignat}
\end{subequations}
where \eqref{eq: em_kkt_1} and \eqref{eq: em_kkt_3} are respectively dual constraints corresponding to primal variables $p_{gt}$ and $\punmet$. \eqref{eq: em_kkt_12} and \eqref{eq: em_kkt_13} are complementary slackness constraints respectively for dual constraints \eqref{eq: em_kkt_1} and \eqref{eq: em_kkt_3} and their corresponding primal variables. \eqref{eq: em_kkt_18} enforce complementary slackness between primal constraints \eqref{eq: em_5} and their corresponding dual variables.

Interestingly, the presence of the load shedding term $\punmet$ in (EM) provides an upper bound for the shadow price $\lambda_{it}$ via the dual problem, as indicated by dual constraints \eqref{eq: em_kkt_3}. Thus, $\cvoll$ can also be viewed as a price cap. This upper bound is necessary as the demand is assumed to be inelastic in the energy market, and thus the electricity price can potentially be infinitely high if there is no price cap. 

Using the KKT conditions, we can prove the following result about a thermal generator's profit in the energy market:
\begin{theorem}\label{th: profit_em}
If the thermal generator $g$ bids truthfully in the energy market, then its profit from the energy market at $t$ is as follows:

(1) If $\lambda^*_{i(g), t} > C^V_g$, then the profit is $(\lambda^*_{i(g), t} - C^V_g)P^\max_g\geq 0$. In particular, if $p_{i(g), t}^{\rm{Unmet} *} > 0$, then $\lambda^*_{i(g), t} = \cvoll$ and the profit is $(\cvoll - C^V_{g})P^\max_g$;

(2) If $\lambda^*_{i(g), t} \leq C^V_g$, then the profit is 0. 
\end{theorem}




\begin{remark}\label{rm: profit_em_renew}
For a renewable generator, since it does not have any variable cost, its profit in the energy market is $\lambda^*_{i(g), t} F^{CF}_{gt} P^\max_g\geq 0$, and the profit equals 0 only if $\lambda^*_{i(g), t} = 0$. This contributes to the observation that wind generators have high energy market profit and low net CONE in our case study.
\end{remark}

A marginal generator in the energy market is the last (i.e., the most expensive) generator to supply electricity at a bus. We are particularly interested in the profit of a thermal generator when it is marginal: 
\begin{proposition}\label{th: marg_gen}
Assuming a thermal generator $\gtilde$ is marginal at bus $i$ and time period $t$:

(1) If the generator does not produce at its full capacity, then its profit from the energy market at $t$ is 0;

(2) If the generator produces at its full capacity, then its profit from the energy market at $t$ is between 0 and $(\cvoll - C^V_{\gtilde})P^\max_{\gtilde}$, and equals $(\cvoll - C^V_{\gtilde})P^\max_{\gtilde}$ when $p^{\rm{Unmet} *}_{i(\gtilde), t} > 0$.
\end{proposition}



With this proposition, we can derive the following useful result for the net CONE of generators that operate below their full capacity, which are usually generators with high variable costs:
\begin{corollary}\label{th: marg_gen_cone}
If a thermal generator $g$ never operates at the full capacity, then its net CONE equals its investment cost.
\end{corollary}

Since the peaker is a thermal plant that has the highest variable costs, it is always the marginal generator when it operates. Therefore, the peaker earns a positive profit in energy market only if it operates at full capacity. In other words, the peaker benefits greatly from capacity scarcity and shortages in the energy market. If the peaker never operates at full capacity, then $\Ccone$ equals the peaker's investment cost.

As a summary, our analysis in Section \ref{ch: cm_profit} and this section provides insights on the benefits and limitations of capacity markets. If there is no capacity market, then a generator is profitable only if its net CONE equals 0, while in practice the majority of generators have a positive net CONE. To ensure the profitability of more generators without the capacity market, we need to increase the total profit of generators in the energy market $\sum_{t\in\T}(\lambda^*_{i(g), t} - C^V_g)^+ P^\max_g$, which means the market clearing price $\lambda^*_{i(g),t}$ needs to be higher in more time periods. This is achieved by setting higher VOLL and having more time periods of scarcity, which creates undesirable price spikes and reduces reliability of the system.
On the other hand, with the capacity market, generators with lower net CONE are profitable and can earn a high profit from the capacity payment. In contrast, the peaker is usually unprofitable, so additional revenue streams may be necessary to keep them in the system.

\section{Capacity Market and Market Power}\label{ch: stackelberg}
We study the market power in joint capacity and energy markets with a leader-follower game framework. The leader is a dominant strategic generator that can submit untruthful bids in both markets, while the other generators are modeled as a competitive fringe and they submit truthful bids. After accepting bids from all generators, the ISO runs auctions to clear the market. The market clearing process is modeled as a follower. This setup with strategic generator(s) as the leader(s) and market clearing as the follower is a commonly-used framework for modeling market power in electricity markets, see for example, \cite{ehrenmann2009comparison} which models strategic generators in energy markets with complex networks, and \cite{zhao2010gaming} which models strategic generators in a joint energy-reserve market. 


In general, there could be several reasons for a firm to become dominant. For example, it may have a larger market share, or it is more efficient, or it has better reputation, or it colludes with other firms to act as a dominant firm \citep{micro2016}. In the energy market, there are usually a few dominant generators that can influence the market outcome, while the other generators with small market share are incentivized to bid truthfully \citep{zhao2010gaming}. The capacity market can also be dominated by a few large generators \citep{schwenen2015strategic}. Our model assumes that there is a single generator that dominates both markets, which is a simplification of real-life strategic behavior in electricity markets.

The goal of our study {\it is not} to provide a comprehensive description of possible strategic behaviors in the capacity and energy markets. Instead, we would like to showcase the market outcome when there is a strategic generator. We demonstrate how strategic bidding behavior may interfere with the intended outcome of the capacity market, which lends insights to the design of market power mitigation policies (Section \ref{ch: stack_cm}). Our framework also serves as a tool for investigating under what conditions the capacity market may mitigate market power in the energy market (Section \ref{ch: stack_both}). In addition, we propose a trilevel optimization model for the leader-follower game in both markets (Section \ref{ch: trilevel}), which enables us to efficiently solve large-scale case studies.

A strategic generator could behave differently when threatened with regulatory intervention \citep{ehrenmann2009comparison}. We do not consider such regulatory threat in this study. Instead, we focus on market-power-related issues in a deregulated setting.

It could be debatable whether a generator's truthful bid in the capacity market is net CONE. If the investment cost is treated as a sunk cost, then a generator's true cost of providing capacity equals the going-forward cost, which corresponds to the fixed O\&M cost in our setup. Yet, if a generator acquired financing through loans and has to repay them over time, or if we consider its value for resale, then it is reasonable not to view the investment cost as sunk \citep{tirole1988theory}. Without further complicating the discussion, in this work we consider net CONE as the truthful bid since it is the long-term marginal cost of capacity.


\subsection{Market Power in Capacity Market}\label{ch: stack_cm}





 In this section, we use the leader-follower game to study capacity market outcomes in the presence of a strategic generator (i.e., the leader), and based on these results, we offer targeted recommendations for market power mitigation. 

Since a generator's total cost in the capacity market (i.e., total net CONE) is a constant $W_1 F^U_1 P^{\max}_1$, the leader's profit-maximization problem in the capacity market is equivalent to its revenue-maximization problem. Consider a leader that does not necessarily bid its true net CONE. To maximize its revenue in the capacity market, the leader can either bid 0 (or any other prices lower than the market clearing price) and sell all qualified capacity, or select an offer price that maximizes its revenue subject to being the marginal supplier. It will never offer a price that is higher than the market clearing price. We provide conditions where the leader deviates from its truthful price and evaluate the subsequent market outcomes. 
Note that in this section, we assume the leader submits untruthful bids only in price and not in quantity, an assumption that is supported by our case study in Section \ref{ch: cm_experiment}. 

Before we proceed, we define some additional notation. Let the leader's index be $g = 1$ and its offer price $w_1$.  Let $\gdot$ and $\gddot$ respectively denote the most expensive generator in $\Ghat\sm\{\ghat\}$ and the least expensive generator in $\mc{G}\sm\Ghat$, i.e., $W_\gdot = \max_{i\in\Ghat\sm\{\ghat\}}W_i$ and $W_\gddot = \min_{i\in\mc{G}\sm\Ghat}W_i$. We use $\epsilon$ to denote a very small monetary value, e.g., 1 cent, which is used by the leader to undercut other bidders. We use the prime symbol ($'$) to denote the outcome of the leader-follower setup when the leader bids 0. For example, $q'_1$ is the leader's sold capacity when it bids 0. Similarly, we use the double prime symbol ($''$) to denote the outcome when the leader selects a bid that maximizes its revenue subject to being the marginal supplier. Additionally, we define $\B = \frac{1}{2}(\Pi^\max - A \sum_{i\in\Ghat''\sm\{1\}}F^U_i P^\max_i)$, a useful expression that appears in Propositions \ref{th: cm_stack_allocated} - \ref{th: cm_stack_unallocated} below. 

The following lemma is useful for proving subsequent results. It categorizes the market outcome when different types of leaders bid 0. It shows that bidding 0 increases revenue for a marginal supplier with either very high or very low qualified capacity, and for an unallocated supplier.
\begin{lemma}\label{th: cm_bid_0}
If the leader bids 0 in the capacity market, we have the following results for its revenue: 

(1) If the leader $g = 1$ is in the set $\Ghat\sm\{\ghat\}$ when bidding truthfully, then its revenue does not change when it bids 0, and the new market clearing price $W_{\ghat'} = W_{\ghat}$.

(2) If the leader $g=1$ is the marginal supplier when bidding truthfully, then the leader increases revenue by bidding 0 when 
(i) $F^U_1 P^\max_1 \geq \frac{\Pi^\max - W_\gdot}{A} - \sum_{i\in\Ghat\sm\{1\}} F^U_i P^\max_i$ and $F^U_1 P^\max_1 > \frac{W_1}{W_{\ghat'}}\left(\frac{\Pi^\max - W_1}{A} - \sum_{i\in\Ghat\sm\{1\}} F^U_i P^\max_i\right)$, in which case the new market clearing price satisfies $W_{\ghat'} \leq W_\gdot$; 
(ii) $F^U_1 P^\max_1 < \frac{\Pi^\max - W_\gdot}{A} - \sum_{i\in\Ghat\sm\{1\}} F^U_i P^\max_i$, in which case the new market clearing price satisfies $W_{\ghat'} \geq W_\gddot$.

(3) If the leader $g=1$ is in the set $\G\sm\Ghat$ when bidding truthfully, then it generates more revenue when bidding 0, and the new market clearing price satisfies $W_{\ghat'} \leq W_\ghat$.
\end{lemma}

One insight from Lemma \ref{th: cm_bid_0} is that, if the marginal supplier chooses to bid 0, then the market clearing price is more likely to increase when the demand is more elastic (i.e., $A$ is lower), and otherwise it is more likely to decrease. Note that a larger ICAP requirement $(1+\Gamma)D^{Peak}$ contributes to a lower $A$.

We now investigate under what conditions bidding untruthfully benefits the leader. We consider 3 different cases depending on whether the leader is allocated and whether it is the marginal supplier when bidding truthfully. Those results are presented in Propositions \ref{th: cm_stack_allocated} - \ref{th: cm_stack_unallocated}.  

Note that if the leader bids untruthfully and becomes the marginal supplier, there could be multiple possibilities for the updated set of allocated generators, denoted by $\Ghat''$, that could increase the leader's revenue. For example, if the leader has a large qualified capacity and is in $\Ghat\sm\{\ghat\}$ when bidding truthfully, and if both $F^U_\ghat P^\max_\ghat$ and $F^U_\gddot P^\max_\gddot$ are not too high, then the leader may become marginal by bidding $w''_1\in [W_\ghat, W_\gddot)$, in which case $\Ghat'' = \Ghat$, or by bidding a price that is slightly higher than $W_\gddot$, and $\Ghat''=\Ghat\cup\{\gddot\}$. We denote the set of all possibilities of $\Ghat''$ as $\Gset$.

\begin{proposition}\label{th: cm_stack_allocated}
If the leader $g = 1$ is in the set $\Ghat\sm\{\ghat\}$ when bidding truthfully, then it can increase the revenue by bidding a price that is different from its net CONE and become the marginal supplier if $\exists \bar{\Gset}$ such that $\forall \Ghat''\in\bar{\Gset}$, the corresponding $B > W_{\ghat}$, and if either of the following is true:

(i) $\max_{\Ghat''\in\bar{\Gset}, B < W_{\gddot''}}B>\sqrt{A W_{\ghat}F^U_1 P^\max_1}$;

(ii) $\max_{\Ghat''\in\bar{\Gset}} W_{\gddot''} \left(\frac{\Pi^\max-W_{\gddot''}}{A}-\sum_{i\in\Ghat''\sm\{1\}}F^U_i P^\max_i\right) > W_\ghat F^U_1 P^\max_1$.
\end{proposition}

Intuitively, this result shows that if a leader is allocated and not marginal, then it increases the revenue from the capacity market by bidding a higher price and becoming the marginal supplier, if the difference between $W_{\ghat}$ and $W_{\gddot''}$ is large, and the total qualified capacity of other allocated suppliers $\sum_{i\in\Ghat''\sm\{1\}}F^U_i P^\max_i$ is not very high, and itself has a low qualified capacity, and the demand is more elastic. Therefore, such a leader is likely to be truthful in a denser market or a market with relatively high demand, or if it has a very high qualified capacity, or if the demand is less elastic. 


In Proposition \ref{th: cm_stack_allocated} we implicitly assume that $\gddot$ does exist, i.e., the marginal generator is not the peaker. Otherwise, the marginal generator has a net CONE equals $\Ccone$, and when the leader bids higher than $\Ccone$ the market will clear below the capacity requirement.

\begin{proposition}\label{th: cm_stack_marginal}
    If the leader $g=1$ is the marginal supplier when bidding truthfully, then it could increase the revenue by either bidding 0 or bidding a different price as the marginal supplier. The leader prefers to still be the marginal supplier if $\exists \bar{\Gset}$ such that $\forall \Ghat''\in\bar{\Gset}$, the corresponding $B > W_{\ghat'}$, and if either of the following is true:
    
    (i) $\max_{\Ghat''\in\bar{\Gset}, B < W_{\gddot''}}B>\sqrt{A W_{\ghat'}F^U_1 P^\max_1}$;
    
    (ii) $\max_{\Ghat''\in\bar{\Gset}} W_{\gddot''} \left(\frac{\Pi^\max-W_{\gddot''}}{A}-\sum_{i\in\Ghat''\sm\{1\}}F^U_i P^\max_i\right) > W_{\ghat'} F^U_1 P^\max_1$.

    The leader's strategic offer price $w''_1$ equals either $B$ or $W_{\gddot''}$ depending on which price leads to a higher revenue. The leader bids truthfully only if $w''_1 = W_1$.
\end{proposition}
    
    



Since the chance that $w''_1$ equals exactly $W_1$ is small, the leader is unlikely to be truthful when it is the marginal generator. 

Similar to Proposition \ref{th: cm_stack_allocated}, this proposition shows that if the leader is the marginal generator when bidding truthfully, then the leader is likely to bid 0 in a denser market or a market with relatively high demand, or if it has a very high qualified capacity, or if the demand is less elastic. Otherwise, it is more likely to remain the marginal generator. 

\begin{proposition}\label{th: cm_stack_unallocated}
If the leader $g=1$ is in the set $\mc{G}\sm\Ghat$ when bidding truthfully, then it always earns a higher revenue when bidding untruthfully. More specifically, the leader maximizes its revenue either by bidding 0 or by becoming the marginal supplier. It becomes the marginal supplier if $F^U_1 P^\max_1 \geq \frac{\Pi^{\max}-W_{\ghat}}{A} - \sum_{i\in\Ghat\sm\{\ghat\}} F^U_i P^\max_i$ and $\exists \bar{\Gset}$ such that $\forall \Ghat''\in\bar{\Gset}$, the corresponding $B > W_{\ghat'}$, and if either of the following is true:

(i) $\max_{\Ghat''\in\bar{\Gset}, B < W_{\gddot''}}B>\sqrt{A W_{\ghat'}F^U_1 P^\max_1}$;

(ii) $\max_{\Ghat''\in\bar{\Gset}} W_{\gddot''} \left(\frac{\Pi^\max-W_{\gddot''}}{A}-\sum_{i\in\Ghat''\sm\{1\}}F^U_i P^\max_i\right) > W_{\ghat'} F^U_1 P^\max_1$.
\end{proposition}

The conditions in this result are very similar to those in Proposition \ref{th: cm_stack_marginal}, except for an additional lower bound on $F^U_1 P^\max_1$. Therefore, if the leader is unallocated when bidding truthfully, it is likely to bid 0 in a denser market or a market with relatively high demand, or if it has either a very high or very low qualified capacity, or if the demand is less elastic. Otherwise, it is more likely to become the marginal generator.

In contrast to the common belief that a supplier with a larger market share is more likely to be untruthful, our findings in Proposition \ref{th: cm_stack_allocated}-\ref{th: cm_stack_unallocated} indicate that the effect of qualified capacity on bidding behavior can be complicated by market dynamics and generator attributes. Moreover, a small qualified capacity could also induce strategic bidding behaviors.

We are also interested in how the other generators and the consumers are affected by the behavior of the leader. If the leader's strategic behavior (e.g. bidding 0) leads to a lower market clearing price, then the revenue for all other generators either decreases or remains the same. This is because when the market clearing price decreases, all previously unallocated and some of the previously allocated generators are unallocated and have no revenue. For the generators that remain allocated, they sell capacity at a lower price. On the other hand, the consumer surplus, which corresponds to the area of the triangle shaded with blue in Figure \ref{fig: updated_market_eq}, increases with a decrease in the market clearing price. Similarly, if the updated market clearing price increases, then the revenue of all other generators either increases or remains the same, and the consumer surplus decreases.

It has been observed that in the capacity markets of PJM and ISO-NE, the cleared capacity is often well above the target quantity \citep{cap_report_2018}. According to Proposition \ref{th: excess_cap}, such excessive capacity procurement is caused by a low market clearing price $\pi = W_{\ghat}$. In terms of our leader-follower setting, a drop in market clearing price happens when the leader is either marginal or unallocated, and maximizes its revenue by bidding a lower price than its net CONE. We also observe this in some of our numerical examples in Section \ref{ch: cm_experiment}, where leaders bid lower prices to increase their capacity market revenue. The resulting low clearing price increases consumer surplus, but it also decreases the revenue of some generators and thus may undermine power system reliability in the long run.

\begin{remark}[Market Power Mitigation]
In sum, in a market with high demand, there are more fully allocated generators, and thus a strategic generator is more likely to bid truthfully. In such a market, or if the market is dense or the demand is less elastic, a marginal or unallocated strategic generator is likely to bid 0, and thus the policy maker may need to enforce market power mitigation policy such as price floors (we will address this shortly) for some strategic generators with high net CONE to ensure truthful bidding. In contrast, when the market is sparse, has high redundant capacity, and with more elastic demand, the strategic generator is more likely to be a price setter. It is also more likely to become unallocated. In this case, the policy maker could impose both price floors and price caps for generators with relatively high net CONE, and price caps for low net CONE generators with a low qualified capacity.
\end{remark}


In practice, to mitigate the market power in capacity markets that leads to low bidding and market clearing prices, ISOs impose price floors for certain types of generators and require them to bid no less than their net CONE. Such rules, called Minimum Offer Price Rules (MOPR), are criticized by FERC as being discriminatory, and MOPR have undergone several reforms throughout the years. We refer interested readers to \cite{macey2021mopr} for more details on the history of MOPR in capacity market.
\subsection{Capacity Market and Physical Withholding}\label{ch: stack_both}
We now use the leader-follower game framework to investigate how the capacity market could mitigate market power in the energy market. Specifically, we focus on the type of market power in the energy market called {\it physical withholding}. By physical withholding, a generator holds back its available capacity in the energy market, which raises the market clearing price. Physical withholding is undesirable as it creates scarcity and undermines the reliability of power grids. 

In this section, we only consider a thermal generator acting as the leader, as the uncertain and non-dispatchable nature of renewable production makes it hard for a wind generator to bid strategically in the energy market. We assume the leader bids 0 in the capacity market, as this is usually the best strategy for the leader in our case study (see Section \ref{ch: cm_experiment}). Note that if instead the leader is truthful in the capacity market, then the results in this section are still valid as long as the leader is fully allocated in the capacity market. In addition, in this section we do not consider untruthful bids in marginal cost (in Section \ref{ch: combined_experiment} we provide numerical results without this restriction). In practice, ISOs impose market seller offer caps to limit the bidding prices in the energy market, so generators cannot raise the market clearing price with very high bids. We use the prime sign ($'$) to denote the outcome of the strategic setup. 

Since we assume the leader bids 0, without physical withholding the leader's capacity market revenue is $\pi^* P^\max_1$ (we drop $F^U_1$ as it equals 1 for thermal generators). If the leader seeks to increase its revenue in the energy market by physical withholding, it bids $h'_1<P^\max_1$ in the capacity market and gets allocated with $q^{*'}_1 = h'_1$. The corresponding revenue is $\pi^{*'}q^{*'}_1$ with $\pi^{*'} \geq \pi^*$.

In the energy market, by Theorem \ref{th: profit_em} when the leader is truthful its total profit from the energy market is 
$\sum_{t\in\That}(\lambda^*_{i(1),t} - C^V_1)P^\max_1$, where $\That$ is the set of time periods with $\lambda^*_{i(1), t} > C^V_1$, i.e., time periods such that the leader earns a positive profit. 

If we have a connected power network without congestion, i.e., power flow in all transmission lines are below their capacity, then energy market price does not decrease when there is withholding:
\begin{proposition}\label{th: em_lmp}
If the network is connected and without congestion, then the value of $\lambda^*_{it}$ is the same at all buses for each time period. If there is physical withholding by the leader, then $\lambda^{*'}_{i(1), t}\geq \lambda^{*}_{i(1), t}, \forall t\in\mc{T}$.
\end{proposition}

We assume for now that the network is connected and without congestion, and we will discuss what happens without those assumptions later in this section. Since when the leader withholds capacity the new energy market clearing price $\lambdasitp \geq \lambdasit$ and thus $\That\subseteq\Thatp$. 
By withholding, the leader's profit increases because of increased energy market clearing prices. The market clearing price in the capacity market may also increase and raise the revenue. The total addition in capacity market revenue and energy market profit because of higher prices is as follows:
\begin{align}\label{eq: gain_0}
(\pi^{*'}-\pi^*)q^{*'}_1+\sum_{t\in\Thatp\sm\That}\left(\lambda^{*'}_{i(1),t}-C^V_1\right)(q^{*'}_1+v^{*'}_1)+\sum_{t\in\That}\left(\lambda^{*'}_{i(1), t}-\lambda^*_{i(1), t}\right)(q^{*'}_1+v^{*'}_1),
\end{align}
where the first term is the 
increased revenue in the capacity market. The second term is the increase from additional time periods with positive profits in the energy market. The third term is the additional revenue during previously profitable time periods in the energy market.

On the other hand, by withholding the leader could also lose some capacity market revenue and energy market profit because of decreased capacity, which can be calculated as:
\begin{align}\label{eq: loss_0}
\pi^*(P^\max_1-q^{*'}_1)+\sum_{t\in\That}(\lambda^*_{i(1),t}-C^V_1)(P^\max_1-q^{*'}_1-v^{*'}_1),
\end{align}
where the first term is the lost revenue in the capacity market, and the second term is the lost profit in previously profitable time periods.

In our case study (Section \ref{ch: combined_experiment}) we observe that $v^{*'}_1=0$ and $\pi^{*'}=\pi^*$ for all instances, indicating that the leader does not bid additional capacity in the energy market when withholding, and that withholding does not affect capacity market clearing price. Assuming those are always true, we simplify \eqref{eq: gain_0} and \eqref{eq: loss_0} respectively to $\gain$ and $\loss$, as defined below:
\begin{subequations}
    \begin{alignat}{4}
    & \gain := \sum_{t\in\Thatp\sm\That}\left(\lambda^{*'}_{i(1),t}-C^V_1\right)q^{*'}_1+\sum_{t\in\That}\left(\lambda^{*'}_{i(1), t}-\lambda^*_{i(1), t}\right)q^{*'}_1\\
    & \loss := \pi^*(P^\max_1-q^{*'}_1)+\sum_{t\in\That}(\lambda^*_{i(1),t}-C^V_1)(P^{\max}_1-q^{*'}_1).
    \end{alignat}
\end{subequations}
When there is no capacity market, the loss is instead $\lossbar := \sum_{t\in\That}(\lambda^*_{i(1),t}-C^V_1)(P^\max_1-q^{*'}_1)$.

Without the capacity market, the leader withholds its capacity if $\gain > \lossbar$, and the capacity market reduces such withholding if $\gain < \loss$. Therefore, the capacity market is more effective at reducing withholding if the revenue $\pi^*(P^\max_1-q^{*'}_1)$ is high. In particular, $\pi^*$ is high when the peak load is relatively high compared with total available capacity, while a high $P^\max_1-q^{*'}_1$ indicates that a large capacity is withheld. We summarize this result in the following remark:
\begin{remark}\label{rm: cm_withhold}
The capacity market is more effective at reducing physical withholding in the energy market when the system has lower redundant capacity at the peak hour, or when the withheld capacity is higher. 
\end{remark}

Physical withholding can be particularly disruptive if a generator withholds a large amount of capacity in a system that has little backup capacity. Our findings underscore the effectiveness of capacity markets in mitigating withholding during those challenging circumstances.

When there is congestion in the network, it is observed that with a small increase in load the market clearing price at certain buses could decrease. However, in general when there is an increase in load or a decrease in available capacity, the market clearing prices usually do not decrease and thus $\lambda^{*'}_{i(1), t}\geq \lambda^*_{i(1),t}$ is still usually true. In addition, congestion or disconnection in a network reduces available capacity at some buses, which could lead to higher $\lambda^{*'}_{i(1), t}$ and increased $\gain$, and thus higher revenue from the capacity market is needed to prevent withholding.


The capacity market also becomes less effective at mitigating market power if withholding creates more time periods with load shedding, as the value of $\gain$ can increase drastically when the energy market price equals $\cvoll$. In particular, since congestion could reduce available capacity and create more scarcity, there could be more time periods with load shedding if withholding happens in a congested network. Therefore, when combined with measures that mitigate congestion and loading shedding, the capacity market can be more effective in reducing physical withholding.

Finally, note that since $\cvoll$ is higher in energy-only market compared with capacity market (\$9,000/MWh in ERCOT versus \$3,500/MWh in Midwest ISO \citep{caiso_scarcity}), the energy-only market is more vulnerable to physical witholding, as the leader can earn higher profits by creating time periods with load shedding, without concerns about revenue loss in the capacity market.



\subsection{Trilevel Optimization for the Leader-Follower Game}\label{ch: trilevel}
We now propose an optimization model that facilitates large-scale case studies of market power, which allows us to numerically illustrate and extend the analytical results in Section \ref{ch: stack_cm} and Section \ref{ch: stack_both}. More specifically, we propose a trilevel optimization model for a leader-follower game in joint capacity and energy markets, which is based on the capacity market and energy market models developed in Section \ref{ch: model} and Section \ref{ch: combined_market}. In the upper-level, the leader maximizes its profit in both markets. The middle level and lower level are respectively the ISO's social welfare maximization problem in the capacity market and the energy market. Note that when modeling market power in the capacity market, we are constrained to using the QC model \eqref{eq: cm_qc}, while the MIP model \eqref{eq: cm_mip} is not suitable for this application, as the QC model's convexity enables us to represent the market equilibrium using KKT conditions in the middle level. In what follows, we present the leader's and follower's problems in both capacity and energy markets.

{\it Leader's problem in the capacity market }(upper-level): In the capacity market, the leader decides on its offer capacity $h_1$ and offer price $w_1$. Since the cleared capacity $r$ is determined in the follower's problem and is not a decision made by the leader, we use the notation $\bar{r}$ to highlight the fact that it is treated as a constant in the leader's problem. Similarly, the sold capacity of the leader ($q_1$) is also an external decision for the leader's problem. We formulate the leader's profit-maximization problem as follows. Note that since the leader's total cost in the capacity market is a constant $W_1 F^U_1 P^\max_1$, we drop it from the objective:
\begin{subequations}\label{eq: cm_leader}
\begin{alignat}{4}
\hspace{-1cm}\text{(CM-Leader):} \max~~&(- A \bar{r} + \Pi^\max)\bar{q}_1 \label{eq: cm_leader_0} \\
\st~~& h_1 \leq F^{U}_1 P^\max_1\label{eq: cm_leader_1}\\
& h_1, w_1 \geq 0 \label{eq: cm_leader_3},
\end{alignat}
\end{subequations}
where $- A \bar{r} + \Pi^\max$ is from the demand function and equals the market clearing price of the capacity market.

{\it Leader's problem in the energy market }(upper-level): In the energy market, the leader maximizes its profit by choosing its offer capacity $v_{1}$ and offer price $c_1$, as shown below:
\begin{subequations}\label{eq: em_leader}
\begin{alignat}{4}
\text{(EM-Leader):}\max~~&\sum_{t\in\mc{T}} (\bar{\lambda}_{i(1), t} - C_1^v) \bar{p}_{1t}\label{eq: em_leader_0}\\ 
\st~~& \bar{q}_1 + v_{1}\leq P_1^\max\label{eq: em_leader_1}\\
& c_1 \leq \cvoll -\epsilon \label{eq: em_leader_2}\\
& c_1,v_{1} \geq 0.\label{eq: em_leader_3}
\end{alignat}
\end{subequations}

Note that for the leader, variables $\lambda_{i(1), t}$, $p_{1t}$, and $q_1$ are external decisions. We explicitly impose the price cap in \eqref{eq: em_leader_2}, which prevents the leader from bidding an unrealistically high price $c_1$ when the load is very high. $\epsilon$ in \eqref{eq: em_leader_2} is a very small number and we use $\epsilon = 1$ in our experiments.

{\it Follower's problem in the capacity market }(middle-level): The follower's problem in the capacity market models market clearing, and it takes the leader's decisions as given. To formulate the problem, we modify the (CM-QC) formulation by treating $h_1$ and $w_1$ as constants. In addition, all non-strategic generators, which are modelled as a competitive fringe, offer their entire qualified capacity and thus, as indicated by Proposition \ref{th: cm_h_opt}, the social welfare is maximized:
\begin{subequations}\label{eq: cm_follower}
\begin{alignat}{4}
\text{(CM-Follower):} \max~~& - \frac{A}{2} r^2 + (\Pi^\max - \bar{w}_1) q_1 + \sum_{g\in\mc{G}\sm\{1\}} (\Pi^\max - W_g) q_g\hspace{-1.5cm}\label{eq: cm_follower_0}\\
\st~~& \eqref{eq: cm_qc_1}\\
& 0 \leq q_g \leq F^U_g P_g^\max&&  \forall g\in\mc{G}\sm\{1\} \\
& 0\leq q_1\leq  \bar{h}_1. && \label{eq: cm_follower_3}
\end{alignat}
\end{subequations}


{\it Follower's problem in the energy market }(lower-level): The follower's problem in the energy market again takes the ISO's perspective, and treats the leader's decisions and the capacity market outcome as given. It is adapted from the (EM) formulation by treating $c_1$, $q_1$, and $v_{1}$ as constants, and letting all non-strategic generators offer their entire qualified capacity. As explained in Section \ref{ch: stack_both}, we only consider thermal generators as the leader in the energy market:
\begin{subequations}\label{eq: em_follower}
\begin{alignat}{4}
\text{(EM-Follower):} \min~& \sum_{t\in\mc{T}}\left(\bar {c}_1 p_{1t} + \sum_{g\in\Gt\sm\{1\}} C^V_g p_{gt}+\sum_{i\in\mc{N}}\cvoll \punmet\right)\hspace{-2.5cm}\label{eq: em_follower_0}\\
\st~& \eqref{eq: em_1}-\eqref{eq: em_4}, \eqref{eq: em_8}, \eqref{eq: em_9}\\
& p_{1t} \leq \bar{q}_1 + \bar{v}_{1} &~~&\forall t\in\mc{T} \label{eq: em_follower_5} \\
& p_{gt} \leq P_g^\max &~~&\forall g\in\Gt\sm\{1\}, t\in\mc{T} \label{eq: em_follower_6}\\
& p_{gt} = F^{CF}_{gt} P^{\max}_{gt} &&\forall g\in\Gr, t\in\mc{T}.
\end{alignat}
\end{subequations}

We formulate the trilevel model for the leader-follower game by putting together \eqref{eq: cm_leader} - \eqref{eq: em_follower}: The upper-level maximizes the leader's total profit by combining the leader's problems \eqref{eq: cm_leader} and \eqref{eq: em_leader}; The middle-level is the ISO's market equilibrium problem for the capacity market (CM-Follower); The lower-level is the ISO's market equilibrium problem for the energy market (EM-Follower). 

To solve the trilevel model, we reformulate it as a mathematical program with equilibrium constraints (MPEC) problem. In the MPEC reformulation, the upper-level model is kept in its original form, while the middle- and lower-level models are replaced with their KKT conditions:
\begin{subequations}\label{eq: trilevel_stackelberg}
\begin{alignat}{4}
\max~~&\eqref{eq: cm_leader_0} + \eqref{eq: em_leader_0}\\
\st~~& \eqref{eq: cm_leader_1}, \eqref{eq: cm_leader_3}, \eqref{eq: em_leader_1} - \eqref{eq: em_leader_3}\\
& \text{KKT conditions of (CM-Follower) and (EM-Follower)}.
\end{alignat}
\end{subequations}


MPEC is NP-hard due to the nonconvex complementary slackness constraints in KKT conditions, which make it challenging to directly solve large-scale instances such as our NYISO case study using commercial solvers. Instead of solving \eqref{eq: trilevel_stackelberg} directly, we reformulate it into simpler structures that are readily handled by a state-of-the-art bilevel optimization software package. More specifically, we reformulate and decompose the trilevel MPEC model to 2 bilevel models and solve them with BilevelJuMP.jl \citep{diasgarcia2022bileveljump}. Note that to solve the bilevel problem, the BilevelJuMP.jl package uses several ways to reformulate the complementary slackness constraints. We choose the SOS1 mode in the package, which uses special ordered set of type one (SOS1) for reformulation. 

To reformulate the trilevel model to bilevel models, we merge its middle level and lower level. The merge is complicated by the middle-level variable $q_1$ because it is treated as given in the lower level. In what follows, we make $q_1$ an upper-level decision by leveraging our insights into the capacity market, enabling the merge of the middle and lower levels.


As we explain in Section \ref{ch: stack_cm}, there are only two possible outcomes in the capacity market with a revenue-maximizing leader: the leader either bids $w_1=0$, or it becomes the marginal supplier. Therefore, the value of $q_1$ can be controlled by the leader's decision in the upper-level. By reformulating (CM-Leader) to reflect one of the two possible decisions of the leader, $q_1$ essentially becomes an upper-level variable.

More specifically, when the leader bids $w_1 = 0$, it gets fully allocated, and we can add constraints $w_1=0$ and $q_1=h_1$ to (CM-Leader).  On the other hand, if the leader is the marginal supplier, we add the following constraints to (CM-Leader):
\begin{subequations}\label{eq: bilevel_case_2}
    \begin{alignat}{4}
    &w_1=-A\bar{r}+\Pi^\max\label{eq: bilevel_case_2_1}\\
    &q_1 = \bar{r}-\sum_{g\in\G\sm\{1\}}\bar{q}_g\label{eq: bilevel_case_2_2}\\
    &q_1\leq h_1\label{eq: bilevel_case_2_3},
    \end{alignat}
\end{subequations}
where \eqref{eq: bilevel_case_2_1} is the demand curve with the market clearing price equals $w_1$. \eqref{eq: bilevel_case_2_2} defines the allocated capacity $q_1$. \eqref{eq: bilevel_case_2_3} is the upper bound of $q_1$. 

With the reformulation of (CM-Leader), $q_1$ becomes an upper-level variable. Subsequently, we can replace $\bar{q}_1$ with $q_1$ in \eqref{eq: em_leader_1}, replace $q_1$ with $\bar{q}_1$ in \eqref{eq: cm_follower_0}, and remove the redundant constraints \eqref{eq: cm_follower_3} from (CM-Follower). Now the middle- and lower-level problems can be merged. To find the leader's optimal strategy, we solve the bilevel models for both possible outcomes in the capacity market, and the model with the higher optimal objective provides the optimal strategy.

Note that the bilevel reformulations can still be time-consuming to solve for some instances. Inspired by \cite{kleinert2021closing}, we derive a valid inequality based on weak duality to speed up the solution of the bilevel models. Different from the linear valid inequality in \cite{kleinert2021closing}, our valid inequality contains nonlinear terms, as it directly enforces (EM-follower)'s bilinear primal objective  \eqref{eq: em_follower_0} to be no less than its dual objective. This valid inequality is helpful in closing optimality gaps for a few instances in our case study. In Table \ref{tab: ineq_time}, we provide a comparison of computation time without and with valid inequalities (No ineq. and With ineq., respectively) for selected instances in the case study, which shows that the valid inequality can be helpful in speeding up the computation and even in closing the optimality gap for some instances. We refer the readers to Section \ref{ch: experiment} for details of those instances and experiment setup. 
\section{NYISO Case Study}\label{ch: experiment}
We have shown that the capacity market is susceptible to market power, and that it can mitigate physical withholding in the energy market. We now conduct a case study based on the NYISO system~\citep{liang2022operation} to illustrate those results. Furthermore, the case study facilitates a comparison of the diverse outcomes under various types of leaders and demand levels. We also experiment with additional strategic setups in the energy market, such as untruthful offer price and a highly congested network, and investigate the role of capacity market in such settings.

 We use a zonal representation of the NYISO system with 12 zones, which is used in practice for computing energy prices. In the data, each zone is represented by a bus and there are 13 transmission lines in the network. The system is populated with 362 thermal generators and 33 wind farms. The load and wind data are from a day in August. All optimization models are solved using Gurobi 10.0 \citep{gurobi} with the Julia JuMP \citep{DunningHuchetteLubin2017} interface. We use a MacBook Pro with Apple M1 Pro chip, which has an 8-core CPU and 16 GB memory. 
\subsection{Case Study: Market Power in the Capacity Market}\label{ch: cm_experiment}
In this section, we focus on market power in the capacity market. We show the leader types and demand levels that make the capacity market susceptible to market power. Furthermore, we assess the impact of market power by comparing various metrics (including generators' revenue, consumer surplus, and market clearing price) under scenarios both with and without market power.

With market power, we pick different types of generators as the leader, including natural gas (NG), coal, nuclear, residual fuel oil (RFO), hydro, wood, and wind. In Table \ref{tab: cm_lead_para}, we provide the marginal cost, the net CONE, and the qualified capacity of the leaders. Note that in practice, the net CONE values in different ISOs vary depending on parameter assumptions. For example, wind has a low net CONE in ISO-NE and a high net CONE in PJM. Our net CONE values are calculated based on our data. We observe that RFO is the peaker with the highest variable cost, and it also has the highest net CONE, while wind has 0 variable cost, and its net CONE is also 0. 

\begin{table}[htpb]
\centering
\caption{Parameters of the Leaders}
\label{tab: cm_lead_para}
{\small
\begin{tabular}{lrrrrrrr}
\hline
\multicolumn{1}{c}{} & \multicolumn{1}{c}{NG} & \multicolumn{1}{c}{Coal} & \multicolumn{1}{c}{Nuclear} & \multicolumn{1}{c}{RFO} & \multicolumn{1}{c}{Hydro} & \multicolumn{1}{c}{Wood} & \multicolumn{1}{c}{Wind} \\ \hline
$C^V_g$(\$/MWh)                  & 21.1                   & 13.1                     & 4.1                         & 67.6                    & 14.9                      & 35.0                     & 0.0                      \\
$W_g$ (\$/MW-day)                    & 199.6                  & 540.2                    & 810.6                       & 1246.5                  & 420.9                     & 752.1                    & 0.0                      \\
$F^U_g P^\max_g$ (MW)         & 621.0                  & 655.1                    & 1299.0                      & 901.8                   & 250.0                     & 42.1                     & 35.9                     \\ \hline
\end{tabular}}
\end{table}

For different types of generators and various demand levels, we solve MPEC models of the leader-follower game, which are similar to problem \eqref{eq: trilevel_stackelberg} except that they only include capacity market decisions and constraints. For comparison, we also solve (CM-QC) for the setting without market power with the same set of parameters. All instances are solved to optimality within 6 minutes.

\begin{figure}[htbp]
\centering
\begin{subfigure}[t]{0.5\textwidth}
\centering
\includegraphics[width = 0.8\textwidth]{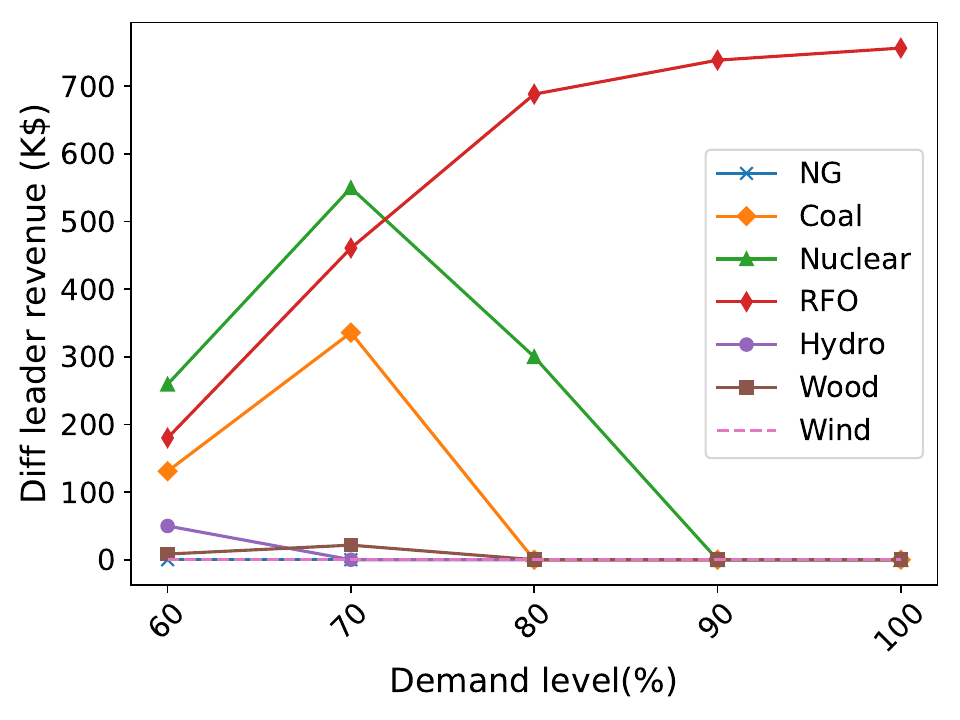}
\caption{}\label{fig: cm_leader_rev}
\end{subfigure}
\hfill
\begin{subfigure}[t]{0.5\textwidth}
\centering
\includegraphics[width = 0.8\textwidth]{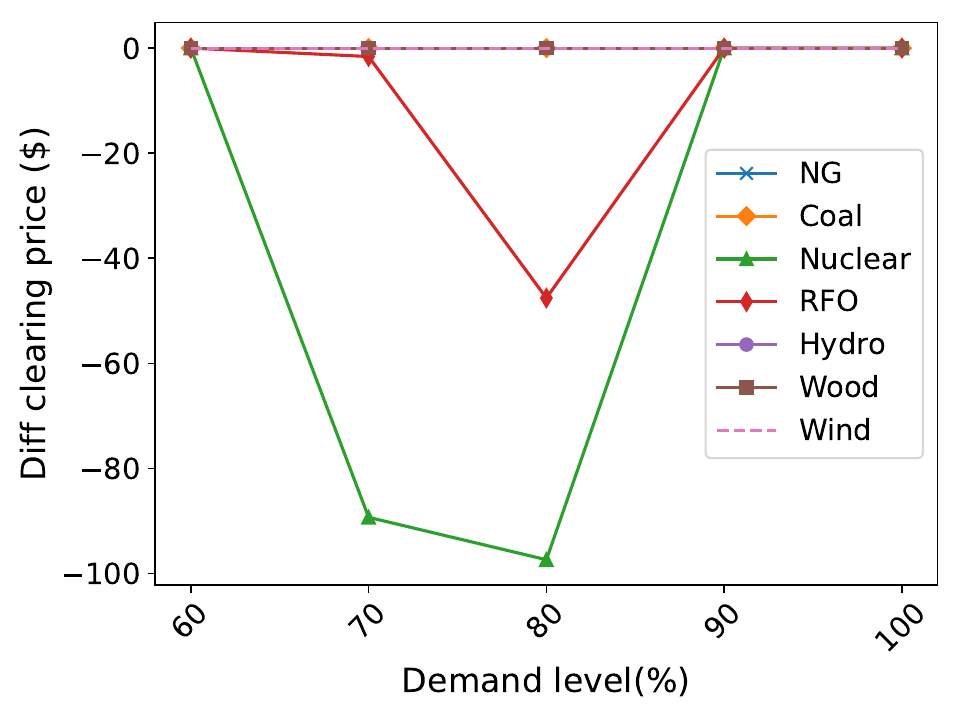}
\caption{}\label{fig: cm_clearing_price}
\end{subfigure}
\caption{\small Difference between strategic setting and truthful setting (the former minus the latter) in the capacity market for (a) leader's capacity market revenue, (b) market clearing price.}
\label{fig: cm_leader+price_diff}
\end{figure}

Figure \ref{fig: cm_leader+price_diff} displays comparisons between the 2 settings for the leader's revenue and the market clearing price. The horizontal axes represent varying levels of demand, from 60\% to 100\% of the original peak demand, increasing in steps of 10\%. Figure \ref{fig: cm_leader_rev} illustrates the difference in leader's revenue between strategic and truthful settings. The generators with higher qualified capacity and higher net CONE (e.g., RFO, Nuclear, Coal) are more likely to benefit from exercising market power. Figure \ref{fig: cm_clearing_price} shows that when the leader exercises market power, the market clearing price either decreases or stays unchanged. If the price is unchanged, it is because the leader breaks ties with other generators having the same net CONE by marginally undercutting them. We observe that nonzero difference, which indicates altered market outcome due to market power, happens less often when the demand is higher, and when the leader has lower net CONE. This can be explained by our results in Section \ref{ch: stack_cm}: generators are more likely to bid truthfully when they are fully allocated, and thus high demand and low net CONE facilitates truthful bidding.

Lemma \ref{th: cm_bid_0} indicates that if the leader is a marginal supplier, the market clearing price could either increase or decrease when the leader bids 0, and the price decreases when the leader has a large qualified capacity. Our observation in the case study is consistent with this result: The nuclear plant has a large qualified capacity and is marginal when demand level is 80\%, and the market clearing price decreases when it bids 0. 

\begin{figure}[htbp]
\centering
\begin{subfigure}[t]{0.5\textwidth}
\centering
\includegraphics[width = 0.8\textwidth]{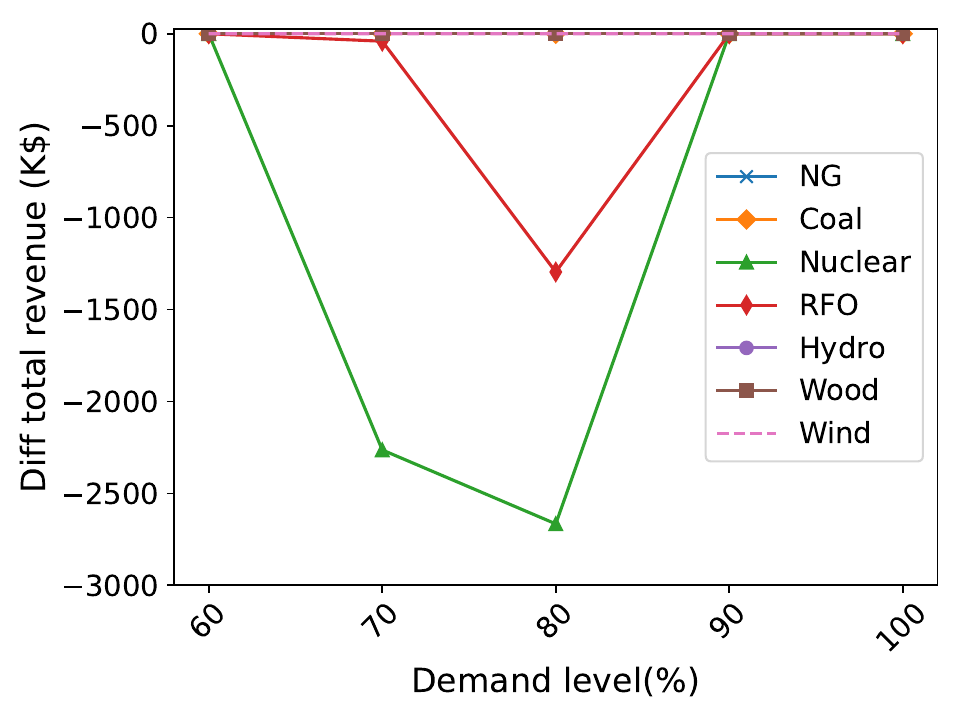}
\caption{}\label{fig: cm_total_rev}
\end{subfigure}
\hfill
\begin{subfigure}[t]{0.5\textwidth}
\centering
\includegraphics[width =0.8\textwidth]{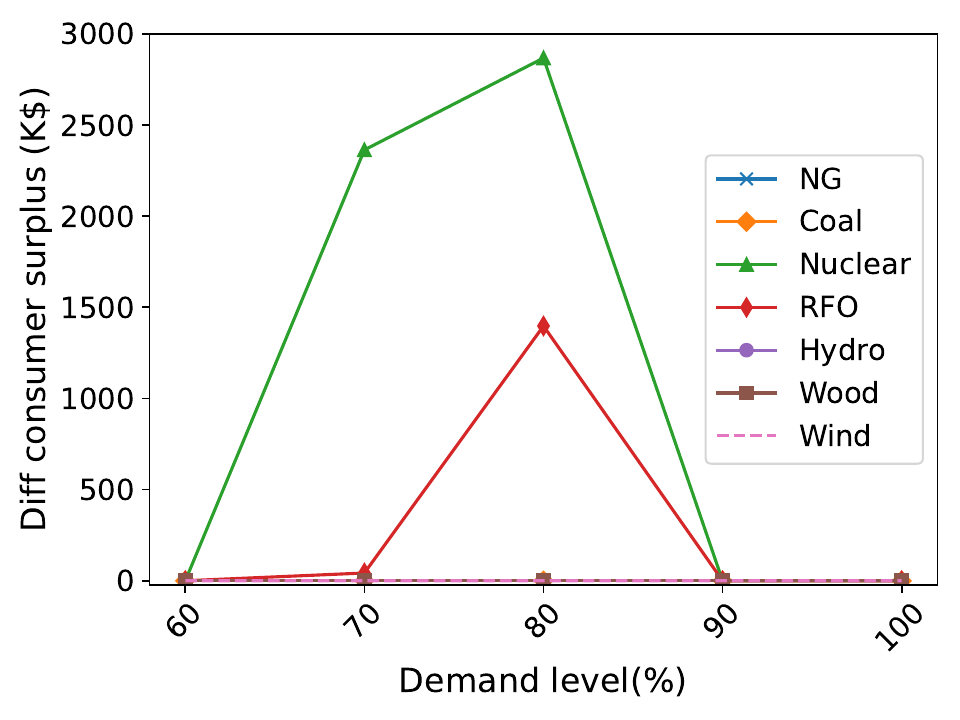}
\caption{}\label{fig: cm_con_spl}
\end{subfigure}
\caption{\small Difference between strategic setting and truthful setting (the former minus the latter) in the capacity market for (a) total revenue of generators, (b) consumer surplus.}
\label{fig: cm_total_rev_con_spl}
\end{figure}

Figure \ref{fig: cm_total_rev} shows the difference in generators' total revenue, while Figure \ref{fig: cm_con_spl} shows the difference in consumer surplus. Both values are directly impacted by the market clearing price and thus show the same pattern of change as Figure \ref{fig: cm_clearing_price}. When the nuclear plant is the leader, it tends to change total revenue and consumer surplus more than other types of leaders. This is probably because the nuclear plant has very large capacity and thus has a larger impact when exercising market power.




\begin{figure}[hbpt]
  \centering
  \includegraphics[width=0.4\textwidth]{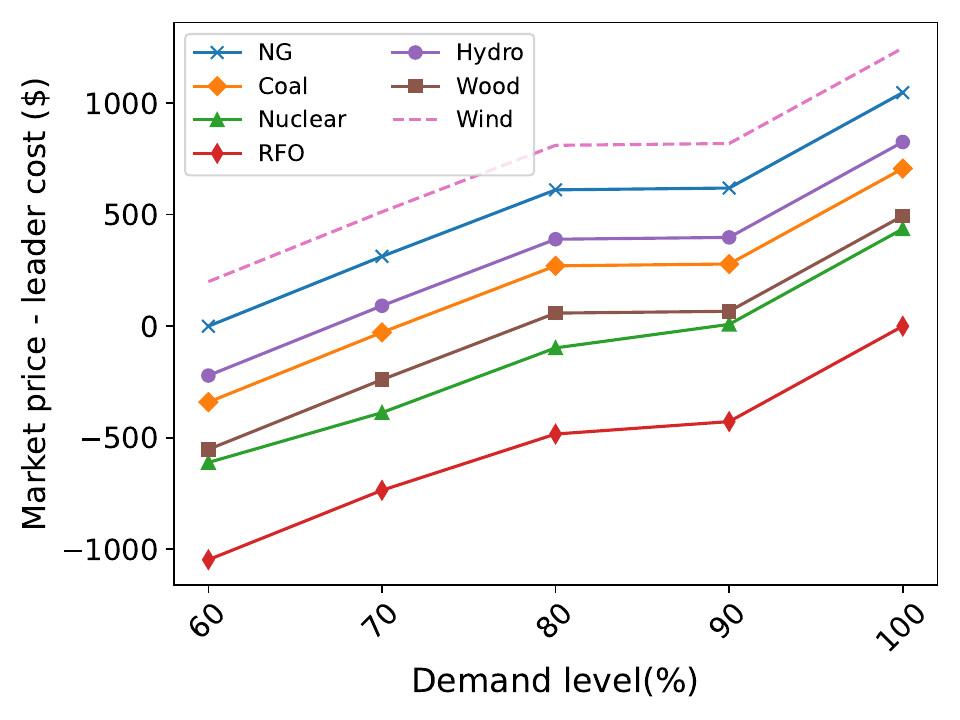}
  \caption{Capacity market clearing price minus leader's net CONE in strategic setting.}
  \label{fig: cm_cone_stack}
\end{figure}

Figure \ref{fig: cm_cone_stack} shows the difference between the market clearing price and the leader's net CONE in the strategic setting, for different leaders and demand levels. Since in our results leaders are always fully allocated when exercising market power, a nonnegative difference indicates that the leader is profitable. Generators with a low net CONE (e.g. wind and NG) are always profitable even if they are not strategic, while the peaker (RFO) is only profitable when the demand is very high. 


Our results suggest that the ISO needs to be vigilant about strategic bids from generators with high net CONE, especially if those generators also have high qualified capacities. When the peak load is low, market mitigation policy (such as price floors) targeting those generators could be used to prevent them from bidding low prices and distorting the market outcome.

\subsection{Case Study: Capacity Market and Market Power in Energy Market}\label{ch: combined_experiment}

We now demonstrate the capacity market's impact on the leader's strategic behavior in the energy market. We measure the market power by comparing the leader's profit from the energy market under different market settings. More specifically, we compare the following settings:

(1) ``S. Both": Strategic setting with both capacity and energy markets.

(2) ``noCM": Strategic setting with only the energy market.

(3) ``True": Truthful setting with both capacity and energy markets.

Similar to the previous section, for each setting we compute the leader's energy market profit for different types of leaders and demand levels. We also consider two congestion levels, where the high congestion level reduces the original transmission line limit (i.e., the normal congestion) by one third. All experiments in this section are run with a 20-minute time limit.

First, we focus on physical withholding and do not consider untruthful bids in marginal cost, which is the same setup as in our analysis for Section \ref{ch: stack_both}. We show the results in Figure \ref{fig: combined_em_profit} for the normal congestion level. We consider different types of thermal generators as the leader, along with demand levels ranging from 60\% to 120\% of original peak load, with increments of 10\%. Figure \ref{fig: combined_both_noCM} reports the difference in leader's profit under ``S. Both" and ``noCM" settings (``S. Both vs. noCM"), which shows the extra profit the leader gets by exercising market power when the capacity market is absent. Figure \ref{fig: combined_stack_true} reports the difference in leader's profit under ``S. Both" and ``True" settings (``S. Both vs. True"), where a non-zero value indicates that the capacity market is unable to completely eliminate all market power. In Section \ref{ec: tables} of the e-companion, we present Table \ref{tab:diff_profit_1} which includes detailed data corresponding to Figure \ref{fig: combined_em_profit}, as well as the results for the same comparisons at a high congestion level.

\begin{figure}[htbp]
\centering
\begin{subfigure}[t]{0.5\textwidth}
\centering
\includegraphics[width = 0.8\textwidth]{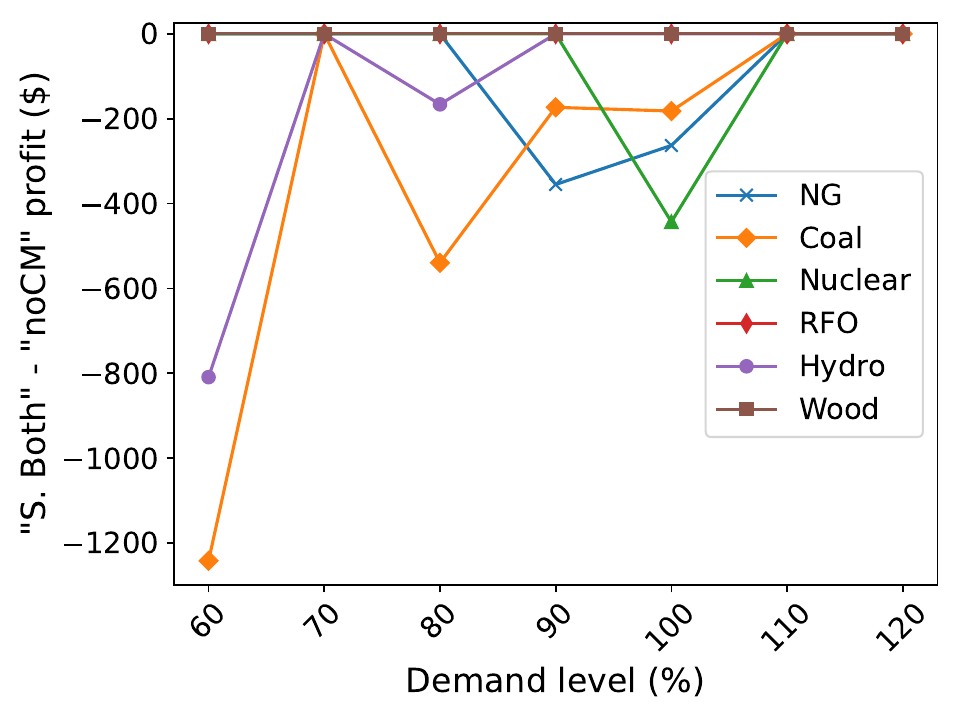}
\caption{}\label{fig: combined_both_noCM}
\end{subfigure}
\hfill
\begin{subfigure}[t]{0.5\textwidth}
\centering
\includegraphics[width = 0.8\textwidth]{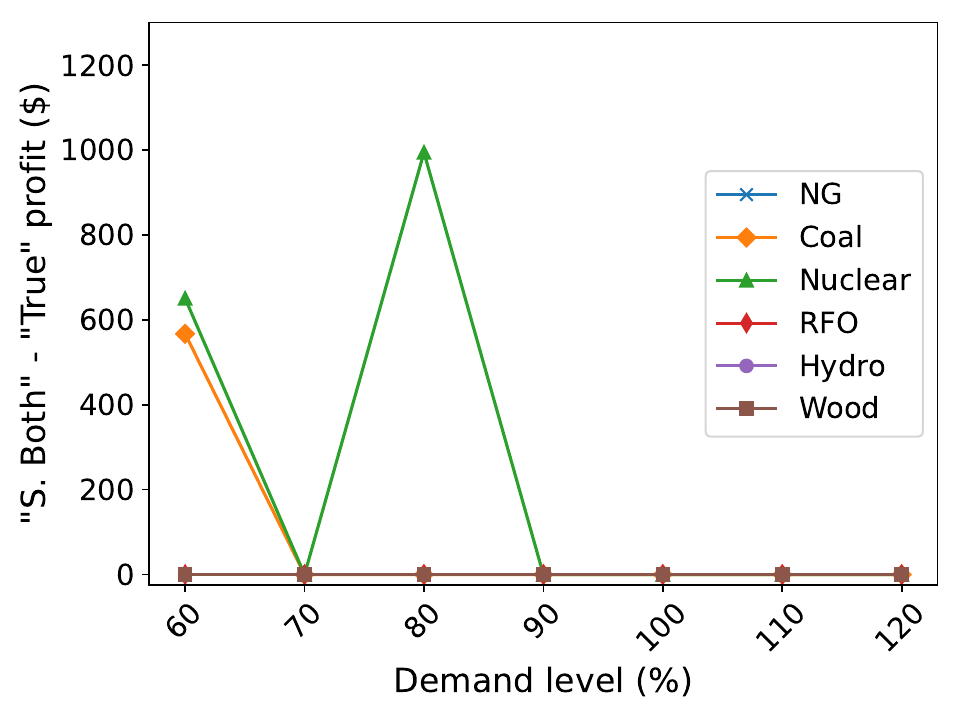}
\caption{}\label{fig: combined_stack_true}
\end{subfigure}
\caption{\small Comparison of the leader's profit from the energy market (with normal congestion) between the following settings (the former minus the latter): (a) Under strategic setting, joint capacity and energy markets vs. only energy market. Non-zero values indicate reduced physical withholding due to capacity market; (b) For joint capacity and energy markets, strategic setting vs. truthful setting. Non-zero values indicate physical withholding.}
\label{fig: combined_em_profit}
\end{figure}

Note that for one of the instances, namely the instance with Hydro being the leader at an 80\% demand level and under the high congestion ``noCM" setting, we use the weak duality-based valid inequality to help close the optimality gap (see Table \ref{tab: ineq_time} in the e-companion). In addition, there are 3 instances (noted with $^*$ in Table \ref{tab:diff_profit_1}) not solved to optimality under ``noCM" setting with high congestion, which lead to possibly underestimated differences between ``S. Both" and ``noCM" settings. Note that for those suboptimal instances, the optimality gaps at termination are usually small. Also, the solver usually quickly identifies high-quality solutions for our instances, with slow convergence mainly due to difficulty in bound verification.

As shown in Figure \ref{fig: combined_both_noCM}, with the capacity market, the leader gets a lower profit from the energy market, which indicates a reduced level of physical withholding. In addition, Figure \ref{fig: combined_stack_true} shows that for most cases, the existence of the capacity market ensures that the leader bids truthfully in capacity in the energy market. The capacity market is more effective in ensuring truthful bidding at higher demand levels, where the leader earns higher capacity market revenue. This observation is consistent with our finding in Remark \ref{rm: cm_withhold}. 

With more congestion in the network, physical withholding becomes more frequent if there is no capacity market, as shown in Table \ref{tab:diff_profit_1}. Also, the leader earns a higher profit via withholding, and more expensive generators, such as Wood and RFO, start to exercise market power. The capacity market is still effective in preventing many cases of withholding.

Next, we relax the assumption that the leader bids truthfully in marginal cost. We include the results in Table \ref{tab:diff_profit_2} of the e-companion, which show that the leader makes more profit by exercising market power. Again, with more congestion the problem of market power becomes more prominent. While the capacity market can still reduce market power, it becomes less effective. 

In sum, the capacity market is effective in reducing market power, especially physical withholding, in the energy market. Yet it does not eliminate market power. Additional measures, such as reducing congestion and incentivizing truthful bids in marginal cost, can be helpful as well.

\section{Conclusion}\label{ch: conclusion}
In this work, we assess the capacity market's impact on generators' revenue and system reliability via a rigorous analytical framework, featuring a new perspective that considers revenue in joint capacity and energy markets, novel equilibrium-based and game-theoretic models depicting market clearing and market power, and innovative solution technologies for large-scale case studies. By utilizing this framework, our analysis delivers valuable insights to both market participants and regulators. We show that the capacity market's effectiveness in enhancing system reliability depends on various factors, including the net CONE of generators, the demand level and elasticity, and the density of the market. Our results confirm that the capacity market is helpful in ensuring resource adequacy. Furthermore, we recommend additional measures that could mitigate market power and aid the capacity market in supporting optimal generation portfolio. 

There are several promising directions for future research. While we have not explored the implications of additional capacity market features, such as the interaction between different capacity auctions, the impact of transmission congestion, and the incorporation of energy storage, it remains essential to understand them. Our approach and findings provide a foundation upon which future endeavors can build to study those features. More broadly, there are a multitude of electricity market policies used in practice, and subjecting them to thorough economic analysis without over-simplifications can greatly enhance our understanding of their implications. By introducing an analytical framework for the capacity market and demonstrating its usefulness, we view our efforts as a step towards this goal. Finally, further research on mechanism design that incentivizes optimal investment and allocation is needed, as it is crucial to the long-term sustainability of electricity markets, as well as other markets characterized by substantial upfront investments. 
\bibliographystyle{informs2014} 
\begingroup
    \setlength{\bibsep}{-1.5pt}
    \linespread{1}\selectfont
\bibliography{reference}
\endgroup

\ECSwitch


\ECHead{e-Companion}

\section{Greedy Algorithm for Clearing the Capacity Market}\label{ch: greedy}
\begin{algorithm}[H]
\SetAlgoLined
{\bf Initialization}: Sort the generators in an ascending order of $W_g$. Let $q = h_1$ and $g = 1$\;
\While{$\frac{\Pi^\max - W_{g}}{A}> q$ \bf{and} $g < |\mc{G}|$}{
	$g = g + 1$\;
	$q = q + h_g$\;
}
\uIf{$\frac{\Pi^\max - W_{g}}{A} \leq q$}{
$\pi^* = W_g$\;
$r^* = \frac{\Pi^\max - W_{g}}{A}$\;
$\ghat = g$\;
$g'=1$\;
    \For{$g'\leq |\G|$}{
        \uIf{$g' < g$}{
        $q_{g'} = h_{g'}$\;}
         \uElseIf{$g' = g$}{
        $q_{g'} = r^* - \sum_{i=1}^{g'-1} h_i$\;}
         \Else{
         $q_{g'} = 0$\;}
         $g' = g'+1$\;}
 {\bf return} $\ghat$, $\pi^*$ and $q_i, \forall i\in\G$\;
 }
\Else{
The capacity market fails to clear\;}
\caption{Greedy Algorithm for Clearing the Capacity Market}\label{alg: greedy}
\end{algorithm}
If the market clears, this algorithm returns the index of the marginal supplier, the market clearing price, and the sold capacity of all suppliers.
\section{Proofs}\label{ec: proofs}
\begin{repeattheorem}[Proposition \ref{th: excess_cap}.]
The proportion that the market clearing quantity exceeds the capacity requirement $\frac{r^* - \Qcap}{\Qcap} = (1 - \frac{W_{\ghat}}{\Ccone})F^E \geq 0$. In particular, if the marginal supplier $\ghat$ is the peaker, then $r^* = \Qcap$.
\end{repeattheorem}
\proof{Proof. of Proposition \ref{th: excess_cap}}
We can directly calculate the proportion: 
\begin{align*}
    \frac{r^* - \Qcap}{\Qcap} &= \frac{\Pi^\max - W_\ghat}{A\Qcap} - 1\\
    & = \frac{(\frac{1+F^E}{F^E}\Ccone- W_\ghat)F^E \Qcap}{\Qcap\Ccone} - 1\\
    & = (1 - \frac{W_\ghat}{\Ccone})F^E,
\end{align*}
where the second equality is because $A = \frac{\Ccone}{F^E \Qcap}$ and $\Pi^\max = \frac{1+ F^E}{F^E}\Ccone$. The last expression is nonnegative because $\Ccone = \max_{g\in\mc{G}} W_g$.

If $\ghat$ is the peaker, then $W_{\ghat} = \Ccone$, thus $\frac{r^* - \Qcap}{\Qcap} = 0 \Rightarrow r^* = \Qcap$.
\Halmos\endproof
\begin{repeattheorem}[Proposition \ref{th: cm_h_opt}.]
There exists an optimal solution for the QC model \eqref{eq: cm_qc} with $h_g = F^U_g P^\max_g, \forall g\in\mc{G}$.
\end{repeattheorem}
\proof{Proof. of Proposition \ref{th: cm_h_opt}}
Let $f^{CM}(\vx)$ denote the objective function in \eqref{eq: cm_qc_0} with $\vx$ representing all variables of (CM-QC) excluding $h_g, \forall g\in\mc{G}$. Let $\mc{F}^{CM}$ be the projected feasible region of variables $\vx$ in \eqref{eq: cm_qc} obtained by replacing the right-hand side of \eqref{eq: cm_qc_3} with $F^U_g P^\max_g$. Then any projected feasible region $\mc{F}^{CM'}$ obtained by replacing the right-hand side of \eqref{eq: cm_qc_3} with values less than $F^U_g P^\max_g, \forall g\in\mc{G}'\subseteq \mc{G}$ and equal $F^U_g P^\max_g, \forall g\in\mc{G}\sm\mc{G}'$ is a subset of $\mc{F}^{CM}$ and we have $\max\{f^{CM}(\vx)|\vx\in\mc{F}^{CM}\} \geq \max\{f^{CM}(\vx)|\vx\in\mc{F}^{CM'}\}$. Therefore, there exists an optimal solution for \eqref{eq: cm_qc} with $\vx\in\mc{F}^{CM}$ and $h_g = F^U_g P^\max_g, \forall g\in\mc{G}$.\Halmos
\endproof
\begin{repeattheorem}[Theorem \ref{th: profit_cm}.]
Let $\Ghat$ be the set of generators that sold their capacity in the capacity market, and let $\ghat$ be the marginal supplier. Assuming the net CONE is correctly estimated, and $W_{\ghat} \neq 0$. Then the profit of generator $g$ is as follows:

(1) If $g\in\Ghat\sm\{\ghat\}$, then its profit equals $(W_{\ghat} - W_g)F^U_g P^\max_g > 0$.

(2) If $g$ is the marginal supplier $\ghat$, then its profit equals $W_{\ghat}(\frac{\Pi^\max - W_\ghat}{A} - \sum_{g\in\Ghat}F^U_g P^\max_g)\leq 0$

(3) If $g\notin\Ghat$, then its profit equals $-W_g F^U_g P^\max_g < 0$
\end{repeattheorem}
\proof{Proof. of Theorem \ref{th: profit_cm}}
Proposition \ref{th: cm_h_opt} indicates that one social welfare-maximization solution is for generators to bid $h_g = F^U_g P^\max_g$. Using the market clearing quantities provided in the greedy algorithm of Section \ref{ch: greedy}, when the market clears the generator $g\in\Ghat\sm\{\ghat\}$ sells its qualified capacity $h_g = F^U_g P^\max_g$, and the marginal generator sells $r^*- \sum_{g=1}^{\ghat-1}h_g=\frac{\Pi^\max - W_\ghat}{A} - \sum_{g\in\Ghat\sm\{\ghat\}}F^U_g P^\max_g$.

For $g\in\Ghat\sm\{\ghat\}$, its revenue from the capacity market is $W_{\ghat}F^U_g P^\max_g$, and its total net CONE is $W_g F^U_g P^\max_g$, which means its profit is $(W_{\ghat} - W_g)F^U_g P^\max_g$. Since $g$ is not marginal, its net CONE is less than $W_\ghat$ (we had assumed that all net CONEs are different). Thus, the profit of $g$ is strictly positive.

For $\ghat$, its profit equals $W_\ghat ( \frac{\Pi^\max - W_\ghat}{A}- \sum_{g\in\Ghat\sm\{\ghat\}}F^U_g P^\max_g) - W_\ghat F^U_\ghat P^\max_\ghat = W_{\ghat}(\frac{\Pi^\max - W_\ghat}{A} - \sum_{g\in\Ghat}F^U_g P^\max_g)$. Since the market clearing quantity is no more than $\sum_{g\in\Ghat}F^U_g P^\max_g$, the profit of $\ghat$ cannot be positive.

For $g\in\mc{G}\sm\Ghat$, it does not sell any capacity and thus does not have any revenue from the capacity market. Therefore, its profit equals the negative total net CONE: $-W_g F^U_g P^\max_g$. Since $W_g > W_\ghat > 0$, this value is strictly negative.
\Halmos\endproof
\begin{repeattheorem}[Corallary \ref{th: peak_rev}.]
The peaker is revenue balanced (i.e., its revenue equals cost) only if it is the marginal generator and $\Qcap = \sum_{g\in\mc{G}}F^U_g P^\max_g$. Otherwise, it always operates at a loss.
\end{repeattheorem}
\proof{Proof. of Corallary \ref{th: peak_rev}}
When the peaker is the marginal generator, $\pi^* = W_{\ghat} = \Ccone$ and $r^* = \Qcap$. Thus, the capacity it sells is $\Qcap - \sum_{g\in\mc{G}\sm\{\ghat\}} F^U_g P^\max_g$, as all other generators have lower net CONEs and should sell all their capacities. The peaker is revenue balanced only when it sells all its capacity, i.e. when $\Qcap - \sum_{g\in\mc{G}\sm\{\ghat\}} F^U_g P^\max_g = F^U_\ghat P^\max_\ghat \Rightarrow \Qcap = \sum_{g\in\mc{G}}F^U_g P^\max_g$. Otherwise, it has a negative profit.

If the peaker is not the marginal generator, then it does not sell any capacity because it has the highest net CONE. Thus, in this case it also has a negative profit.
\Halmos\endproof
\begin{repeattheorem}[Proposition \ref{th: em_h_opt}.]
There exists an optimal solution for \eqref{eq: em} with $v_{g} = P^\max_g - \bar{q}_g, \forall g\in\G$.
\end{repeattheorem}
\proof{Proof. of Proposition \ref{th: em_h_opt}}
The proof is similar to that of Proposition \ref{th: cm_h_opt}. Let $f^{EM}(\vx)$ denote the objective function in \eqref{eq: em_0}  with $\vx$ representing all variables of (EM) excluding $v_g, \forall g\in\mc{G}$. Let $\mc{F}^{EM}$ be the projected feasible region of variables $\vx$ in \eqref{eq: em} obtained by replacing the terms $v_g$ in the right-hand side of \eqref{eq: em_5} and \eqref{eq: em_6} with $P^\max_g - \bar{q}_g$. Then any projected feasible region $\mc{F}^{EM'}$ obtained by replacing $v_g$ in the right-hand side of \eqref{eq: em_5} and \eqref{eq: em_6} with values less than $P^\max_g - \bar{q}_g, \forall g\in\mc{G}'\subseteq \mc{G}$ and equal $P^\max_g - \bar{q}_g, \forall g\in\mc{G}\sm\mc{G}'$ is a subset of $\mc{F}^{EM}$ and we have $\min\{f^{EM}(x)|x\in\mc{F}^{EM}\} \leq \min\{f^{EM}(x)|x\in\mc{F}^{EM'}\}$. Therefore, there exists an optimal solution of (EM) with $\vx\in\mc{F}^{CM}$ and $v_{g} = P^\max_g - \bar{q}_g, \forall g\in\Gt$.
\Halmos\endproof
\begin{repeattheorem}[Proposition \ref{th: em_eq}.]
The optimal prices $(\lambda_{it}^*)_{i\in\mc{N}}$, production levels $(p_{gt}^*)_{g\in\mc{G}}$, load shedding levels $(\punmetstar)_{i\in\mc{N}}$, and the demand $(D_{it})_{i\in\mc{N}}$ form a competitive equilibrium in the energy market at time $t\in\mc{T}$.
\end{repeattheorem}
\proof{Proof. of Proposition \ref{th: em_eq}}
We show that under the prices and quantities $(\lambda_{it}^*)_{i\in\mc{N}}$, $(p_{gt}^*)_{g\in\mc{G}}$, $(\punmetstar)_{i\in\mc{N}}$, and $(D_{it})_{i\in\mc{N}}$, the market is cleared at $t\in\mc{T}$, and each market participant maximizes their utility, which means those prices and quantities constitute a market equilibrium \citep{mas1995microeconomic}. 

First, we show that the market is cleared at $t\in\mc{T}$. Summing up \eqref{eq: em_1} over $i\in\mc{N}$ and move the terms for $\punmet$ to the right-hand side, we get the following equation:
\begin{align*}
\sum_{g\in\mc{G}} p_{gt} = \sum_{i\in\mc{N}}D_{it} - \sum_{i\in\mc{N}}\punmet.
\end{align*}

Note that the network flow terms $f_{ijt}$ are eliminated, as $f_{ijt}=-f_{jit}$ due to \eqref{eq: em_2}. This equation shows that total supply equals total demand (after load shedding $p^{Unmet*}_{it}$) for any optimal solution $p^*_{gt}, \forall g\in\mc{G}$.

Next, we prove that each market participant maximizes its utility. For the demand side, since the demand is inelastic, there is no need to consider utility maximization for the consumers. 
For the supply side, consider thermal generator $g\in\Gt$, which maximizes its producer surplus by solving a profit-maximization problem:
\begin{subequations}\label{eq: ther_profit_max}
\begin{alignat}{4}
\max~& (\lambda^*_{i(g), t} - C^V_g)p_{gt}\\
\st~& \eqref{eq: em_5}, \eqref{eq: em_7}, \eqref{eq: em_8}, \eqref{eq: em_10}~\text{for time $t$}.
\end{alignat}
\end{subequations}
The profit-maximization problem for a renewable generator is formulated similarly.

We would like to show that at $t$, the optimal production level solution for each generator $p^*_{gt}$ from DCOPF is also optimal for the generator's profit-maximization problem. We reformulate (EM) by dualizing the flow balance constraint \eqref{eq: em_1} with its optimal shadow price $\lambda^*_{it}$:
\begin{subequations}\label{eq: em_reform}
\begin{alignat}{4}
\min~~& \sum_{t\in\mc{T}}\left(\sum_{g\in\Gt} C^V_g p_{gt}+\sum_{i\in\mc{N}}\cvoll\punmet\right) +\nonumber\\ &\sum_{i\in\mc{N}}\sum_{t\in\mc{T}}\lambda^*_{it}\left(D_{it} - \left(\sum_{g\in\mc{G}_i} p_{gt} + \punmet + \sum_{(j, i)\in\mc{E}}f_{jit} - \sum_{(i,j)\in\mc{E}} f_{ijt}\right)\right)\label{eq: em_reform_0}\\
\st~~&\eqref{eq: em_2} - \eqref{eq: em_10}
\end{alignat}
\end{subequations}
Since (EM) is a linear program, this reformulation is exact. Reversing the sign of the objective function \eqref{eq: em_reform_0} and rearranging the terms by generators and buses, we turn \eqref{eq: em_reform} to the following maximization problem:
\begin{subequations}\label{eq: em_reform2}
\begin{alignat}{4}
\max~~&\sum_{t\in\mc{T}}\Big(\sum_{g\in\Gt}(\lambda^*_{i(g),t}-C^V_g)p_{gt} + \sum_{g\in\Gr}\lambda^*_{i(g), t}p_{gt}+\sum_{i\in\mc{N}}(\lambda^*_{it} - \cvoll)\punmet\nonumber\\& - \sum_{i\in\mc{N}}\lambda^*_{it}(\sum_{(j,i)\in\mc{E}}f_{jit} - \sum_{(i,j)\in\mc{E}}f_{ijt})\Big)\\\label{eq: em_reform2_0}
\st~~&\eqref{eq: em_2} - \eqref{eq: em_10},
\end{alignat}
\end{subequations}
which can be decomposed by time, generators, and buses. After decomposition, the problem for thermal generator $g\in\Gt$ at time $t$ is the same as the profit-maximization problem \eqref{eq: ther_profit_max}. Therefore, the optimal production level $p^*_{gt}$ is also optimal for the profit-maximization problem. The same can be shown for renewable generators $g\in\Gr$. Thus, under the optimal shadow price $\lambda^*_{it}$, the optimal production solution $p^*_{gt}$ maximizes producer surplus.
\Halmos\endproof
\begin{repeattheorem}[Theorem \ref{th: profit_em}.]
If the thermal generator $g$ bids truthfully in the energy market, then its profit from the energy market at $t$ is as follows:

(1) If $\lambda^*_{i(g), t} > C^V_g$, then the profit is $(\lambda^*_{i(g), t} - C^V_g)P^\max_g\geq 0$. In particular, if $p_{i(g), t}^{\rm{Unmet} *} > 0$, then $\lambda^*_{i(g), t} = \cvoll > C^V_g$ and the profit is $(\cvoll - C^V_{g})P^\max_g$;

(2) If $\lambda^*_{i(g), t} \leq C^V_g$, then the profit is 0. 
\end{repeattheorem}
\proof{Proof. of Theorem \ref{th: profit_em}}
(1) When $\lambda^*_{i(g), t} > C^V_g$, because of the dual constraint \eqref{eq: em_kkt_1}, we must have $\alpha^*_{gt} > 0$. By the complementary slackness constraint \eqref{eq: em_kkt_18}, we have $p^*_{gt} = \bar{q}_g + v^*_{g}$. Since Proposition \ref{th: em_h_opt} shows that when generators bid truthfully $\bar{q}_g + v^*_{g} = P^\max_g$, we have $p^*_{gt}=P^{\max}_g$ and the profit of generator $g$ equals $(\lambda^*_{i(g), t} - C^V_g)p^*_{gt} = (\lambda^*_{i(g), t} - C^V_g)P^\max_g$.

When $p^{\rm{Unmet} *}_{i(g), t} > 0$, because of the complementary slackness constraint \eqref{eq: em_kkt_13}, $\lambda^*_{i(g), t} = \cvoll$. Since $\cvoll$ is higher than the variable cost of any generator, we have $\lambda^*_{i(g), t} = \cvoll > C^V_g$. Using the profit expression we obtained in the previous paragraph, the profit of $g$ is $(\cvoll - C^V_{g})P^\max_g$.

(2) When $\lambda^*_{i(g), t} = C^V_g$, the profit equals 0 because the revenue equals the cost. 

If $\lambda^*_{i(g), t} < C^V_g$, the optimal production level must be 0. This is because if $p^*_{gt} > 0$, then by complementary slackness constraint \eqref{eq: em_kkt_12}, $\lambda^*_{i(g),t} - \alpha^*_{gt} - C^V_g = 0 \Rightarrow \lambda^*_{i(g),t} = \alpha^*_{gt} + C^V_g$, which is impossible because $\lambda^*_{i(g), t} < C^V_g$ and $\alpha^*_{gt} \geq 0$. Therefore, the production level is 0 and thus the profit is also 0.\Halmos\endproof
\begin{repeattheorem}[Proposition \ref{th: marg_gen}.]
Assuming a thermal generator $\gtilde$ is marginal at bus $i$ and time period $t$:

(1) If the generator does not produce at its full capacity, then its profit from the energy market at $t$ is 0;

(2) If the generator produces at its full capacity, then its profit from the energy market at $t$ is between 0 and $(\cvoll - C^V_{\gtilde})P^\max_{\gtilde}$, and equals $(\cvoll - C^V_{\gtilde})P^\max_{\gtilde}$ when $p^{\rm{Unmet} *}_{i(\gtilde), t} > 0$. 
\end{repeattheorem}
\proof{Proof. of Proposition \ref{th: marg_gen}}
(1) Since the thermal generator $\gtilde$ is a marginal generator, $p^*_{\gtilde t} > 0$. Due to the complementary slackness constraint \eqref{eq: em_kkt_12}, we have
\begin{align}\label{eq: equation}
    \lambda^*_{i(\gtilde), t} - \alpha^*_{\gtilde t} - C^V_\gtilde = 0. 
\end{align}

If the generator does not produce at its full capacity, $p^*_{\gtilde t} < P^\max_\gtilde$. Thus, according to Proposition \ref{th: em_h_opt}, there exist optimal solutions for $p_{\gtilde t}$ and $v_{\gtilde}$ such that $p^*_{\gtilde t} < \bar{q}_\gtilde + v^*_{\gtilde} = P^\max_\gtilde$. Because of complementary slackness constraints \eqref{eq: em_kkt_18}, $\alpha^*_{\gtilde t} = 0$ and equation \eqref{eq: equation} becomes $\lambda^*_{i(\gtilde), t} = C^V_\gtilde$. Therefore, according to Theorem \ref{th: profit_em} the profit for generator $\gtilde$ at time period $t$ is 0.

(2) Now consider the generator producing at its full capacity. Because of $\alpha^*_{\gtilde t}\geq 0$ and equation \eqref{eq: equation}, we have $\lambda^*_{i(\gtilde),t} = C^V_\gtilde + \alpha^*_{\gtilde t} \geq C^V_\gtilde$. Because of \eqref{eq: em_kkt_3}, we also have $\lambda^*_{i(\gtilde), t} \leq \cvoll$. Thus, $\lambda^*_{i(\gtilde),t}\in [C^V_\gtilde, \cvoll]$ and $\gtilde$'s profit is no less than 0 and no more than $(\cvoll - C^V_{\gtilde})P^\max_{\gtilde}$. 

If $p^{\rm{Unmet} *}_{i(\gtilde), t} > 0$, then because of Theorem \ref{th: profit_em} the profit of $\gtilde$ equals $(\cvoll - C^V_{\gtilde})P^\max_{\gtilde}$.\Halmos\endproof
\begin{repeattheorem}[Corollary \ref{th: marg_gen_cone}.]
If a thermal generator $g$ never operates at the full capacity, then its net CONE equals its investment cost.
\end{repeattheorem}
\proof{Proof. of Corollary \ref{th: marg_gen_cone}}
If the thermal generator $g$ never operates at its full capacity, then either it generates no electricity and gets no income, or its production is less than its capacity. In the latter case, $g$ must be a marginal generator that has the highest variable costs among all operating generators at the bus. Otherwise, the generation from a more costly generator could have been fulfilled by $g$ at a lower cost. According to Proposition \ref{th: marg_gen}, in this case generator $g$ also gets 0 profit. Therefore, generator $g$ is revenue-balanced and gets 0 profit from the energy market. By the definition of net CONE in \eqref{eq: def_cone}, the net CONE of $g$ equals its investment cost.
\Halmos\endproof
\begin{repeattheorem}[Lemma \ref{th: cm_bid_0}.]
If the leader bids 0 in the capacity market, we have the following results for its revenue: 

(1) If the leader $g = 1$ is in the set $\Ghat\sm\{\ghat\}$ when bidding truthfully, then its revenue does not change when it bids 0, and the new market clearing price $W_{\ghat'} = W_{\ghat}$.

(2) If the leader $g=1$ is the marginal supplier when bidding truthfully, then the leader increases revenue by bidding 0 when 
(i) $F^U_1 P^\max_1 \geq \frac{\Pi^\max - W_\gdot}{A} - \sum_{i\in\Ghat\sm\{1\}} F^U_i P^\max_i$ and $F^U_1 P^\max_1 > \frac{W_1}{W_{\ghat'}}\left(\frac{\Pi^\max - W_1}{A} - \sum_{i\in\Ghat\sm\{1\}} F^U_i P^\max_i\right)$, in which case the new market clearing price satisfies $W_{\ghat'} \leq W_\gdot$; 
(ii) $F^U_1 P^\max_1 < \frac{\Pi^\max - W_\gdot}{A} - \sum_{i\in\Ghat\sm\{1\}} F^U_i P^\max_i$, in which case the new market clearing price satisfies $W_{\ghat'} \geq W_\gddot$.

(3) If the leader $g=1$ is in the set $\G\sm\Ghat$ when bidding truthfully, then it generates more revenue when bidding 0, and the new market clearing price satisfies $W_{\ghat'} \leq W_\ghat$.
\end{repeattheorem}
\proof{Proof. of Lemma \ref{th: cm_bid_0}}
    (1) When the leader is an allocated non-marginal supplier when bidding truthfully, it sells its entire qualified capacity at $W_\ghat$. If instead it bids 0, its allocated capacity stays the same, and the market clearing price does not changed. Thus, its revenue does not change.

    (2) If the leader is the marginal supplier when bidding truthfully, then when it bids 0 instead, the new market clearing price $W_{\ghat'}$ is either no more than $W_\gdot$ or no less than $W_\gddot$. We are interested in finding out when the new market clearing price leads to increased revenue for the leader.

    First, we find the condition under which $W_{\ghat'} \leq W_{\gdot}$. Consider when $W_{\ghat'} = W_\gdot$, the new market clearing quantity $Q^{\rm{Sold}'} = \frac{\Pi^\max - W_\gdot}{A}$, and the updated set of allocated suppliers $\Ghat'=\Ghat$. To clear the market at the price $W_\gdot$, the market clearing quantity $Q^{\rm{Sold}'}$ should be in the interval $(\sum_{i\in\Ghat\sm\{\gdot\}}F^U_i P^{\max}_i, \sum_{i\in\Ghat}F^U_i P^{\max}_i]$. If $Q^{\rm{Sold}'}>\sum_{i\in\Ghat}F^U_i P^{\max}_i$, then the market clearing price is above $W_\gdot$; Else if $Q^{\rm{Sold}'}\leq\sum_{i\in\Ghat\sm\{\gdot\}}F^U_i P^{\max}_i$, then the market clearing price is below $W_\gdot$. Therefore, $W_{\ghat'} \leq W_{\gdot}$ if and only if $Q^{\rm{Sold}'} = \frac{\Pi^\max - W_\gdot}{A} \leq \sum_{i\in\Ghat} F^U_i P^\max_i \Rightarrow F^U_1 P^\max_1 \geq \frac{\Pi^\max - W_\gdot}{A} - \sum_{i\in\Ghat\sm\{1\}} F^U_i P^\max_i$.
    
    Now let $W_{\ghat'} \leq W_{\gdot}$. If the leader increases its revenue by bidding 0, then its updated revenue $W_{\ghat'} F^U_1 P^\max_1 > W_1 \left(\frac{\Pi^\max - W_1}{A} - \sum_{i\in\Ghat\sm\{1\}}F^U_i P^\max_i\right)$, with the right-hand side of the inequality being the revenue of the leader when bidding truthfully. Thus, $F^U_1 P^\max_1 > \frac{W_1}{W_{\ghat'}} \left(\frac{\Pi^\max - W_1}{A} - \sum_{i\in\Ghat\sm\{1\}}F^U_i P^\max_i\right)$. 

    On the other hand, when $W_{\ghat'} \geq W_\gddot$, we have $W_{\ghat'}> W_\ghat = W_1$. In addition, the leader is fully allocated, and thus $q'_1 = F^U_1 P^\max_1 \geq q_1$. Therefore, the leader's revenue when bidding 0, $W_{\ghat'} q'_1$, is always higher than its original revenue $W_1 q_1$.

    (3) If the leader is unallocated when bidding truthfully, then it does not get any revenue. If instead it bids 0, then the updated market clearing price $W_{\ghat'}\in (0, W_\ghat]$. Thus, the leader's updated revenue $W_{\ghat'}F^U_1 P^\max_1 > 0$.
\Halmos\endproof
\begin{repeattheorem}[Proposition \ref{th: cm_stack_allocated}.]
If the leader $g = 1$ is in the set $\Ghat\sm\{\ghat\}$ when bidding truthfully, then it can increase the revenue by bidding a price that is different from its net CONE and become the marginal supplier if $\exists \bar{\Gset}$ such that $\forall \Ghat''\in\bar{\Gset}$, the corresponding $B > W_{\ghat}$, and if either of the following is true:

(i) $\max_{\Ghat''\in\bar{\Gset}, B < W_{\gddot''}}B>\sqrt{A W_{\ghat}F^U_1 P^\max_1}$;

(ii) $\max_{\Ghat''\in\bar{\Gset}} W_{\gddot''} \left(\frac{\Pi^\max-W_{\gddot''}}{A}-\sum_{i\in\Ghat''\sm\{1\}}F^U_i P^\max_i\right) > W_\ghat F^U_1 P^\max_1$.
\end{repeattheorem}
\proof{Proof. of Proposition \ref{th: cm_stack_allocated}}
    As shown in Lemma \ref{th: cm_bid_0}, the leader's revenue does not change when bidding 0. Therefore, when either bidding truthfully or bidding 0, the leader's revenue equals $W_\ghat F^U_1 P^\max_1$. It has the incentive to deviate from the truthful bid only if it becomes a marginal supplier and obtains a higher revenue. When it becomes the marginal generator, its sold capacity is no more than $F^U_1 P^\max_1$, so to increase revenue it must bid a price $w''_1 > W_\ghat$. 

    The capacity that the leader sells when it becomes the marginal supplier is $q''_1 = \qsoldp_1 - \sum_{i\in\Ghat''\sm\{1\}} F^U_i P^\max_i$, and we have the following expression for the leader's revenue:
    \begin{align}\label{eq: rev_leader_marginal}
        w''_1 q''_1 &= w''_1 \left(\frac{\Pi^\max - w''_1}{A} - \sum_{i\in\Ghat''\sm\{1\}} F^U_i P^\max_i\right)\nonumber\\
        &= -\frac{1}{A}(w''_1 - B)^2 + \frac{B^2}{A}.
    \end{align}
     Note that because the set of allocated generators $\Ghat''$ depends on $w''_1$, \eqref{eq: rev_leader_marginal} is a piecewise concave quadratic function. For each $\Ghat''\in\Gset$, the maximizer for the quadratic function is $w''_1 = B$ when ignoring the domain of $w''_1$. The corresponding maximum is $\frac{B^2}{A}$. Note that the value of $B$ depends on the set $\Ghat''$, and it decreases as $\Ghat''$ includes more generators.

    The shape of the piecewise concave quadratic function \eqref{eq: rev_leader_marginal} depends on how $B$ compares with the bidding prices of generators. We illustrate different shapes of \eqref{eq: rev_leader_marginal} in Figure \ref{fig: piecewise}. In Figure \ref{fig: piecewise_1}, $B\leq W_{\ghat}$. Since the value of $B$ decreases as more generators are included in $\Ghat''$, if the leader can bid higher than $W_\gddot$ and gets allocated, then its revenue will continue a decreasing trend, as the stationary point of the corresponding quadratic revenue function moves further left. In Figure \ref{fig: piecewise_2} $B\in (W_\ghat, W_\gddot]$ and in Figure \ref{fig: piecewise_3} $B$ is greater than $W_\gddot$ (and is no more than the net CONE of the next expensive generator to $\gddot$).

    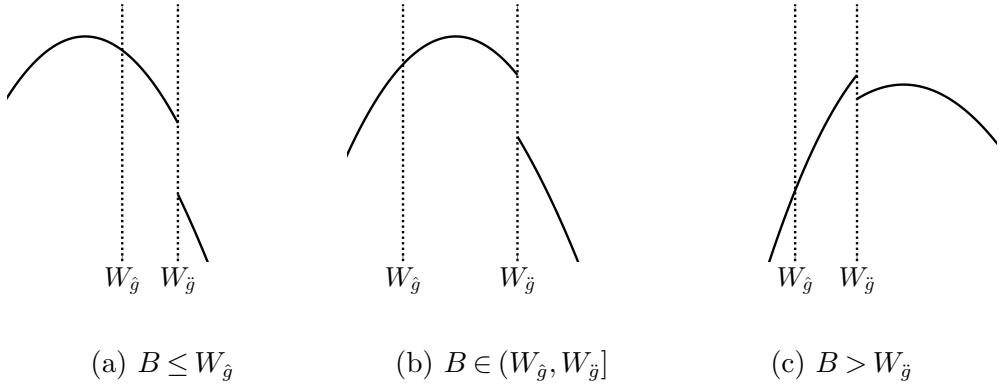
\begin{figure}[bthp]
    \centering
        \begin{subfigure}[t]{0.25\textwidth}
        \centering
        \begin{tikzpicture}[scale=0.6]
\begin{axis}[
    xmin=-2.5, xmax=7.5,
    ymin=1, ymax=9,
    xtick=\empty,
    ytick=\empty,
    xticklabels={}, yticklabels={},
    axis line style={draw=none},
    extra x ticks={1.2, 3},
    extra x tick labels={$W_{\ghat}$, $W_{\gddot}$},
    tick label style={font=\Large},
]

\addplot[domain=-5:3, samples=100, line width=1.5pt] {-0.3*x^2+8};

\addplot[domain=3:6, samples=100, line width=1.5pt] {-0.2*(x+2)^2 + 8.08};

    \draw[dotted, line width=1.5pt] (1.2,1) -- (1.2,9);
    \draw[dotted, line width=1.5pt] (3,1) -- (3,9);

\end{axis}
\end{tikzpicture}
        \caption{$B \leq W_\ghat$}
        \label{fig: piecewise_1}
    \end{subfigure}
        ~
        \begin{subfigure}[t]{0.25\textwidth}
        \centering
        \begin{tikzpicture}[scale=0.6]
\begin{axis}[
    xmin=-3.5, xmax=6.5,
    ymin=1, ymax=9,
    xtick=\empty,
    ytick=\empty,
    xticklabels={}, yticklabels={},
    axis line style={draw=none},
    extra x ticks={-1.7, 2},
    extra x tick labels={$W_{\ghat}$, $W_{\gddot}$},
    tick label style={font=\Large},
]

\addplot[domain=-5:2, samples=100, line width=1.5pt] {-0.3*x^2+8};

\addplot[domain=2:6, samples=100, line width=1.5pt] {-0.2*(x+2)^2 + 8.08};

    \draw[dotted, line width=1.5pt] (-1.7,1) -- (-1.7,9);
    \draw[dotted, line width=1.5pt] (2,1) -- (2,9);

\end{axis}
\end{tikzpicture}
        \caption{$B\in (W_\ghat, W_\gddot]$}
        \label{fig: piecewise_2}
        \end{subfigure}
        ~
        \begin{subfigure}[t]{0.25\textwidth}
        \centering
         \begin{tikzpicture}[scale=0.6]
\begin{axis}[
    xmin=-7.5, xmax=2.5,
    ymin=1, ymax=9,
    xtick=\empty,
    ytick=\empty,
    xticklabels={}, yticklabels={},
    axis line style={draw=none},
    extra x ticks={-4, -2},
    extra x tick labels={$W_{\ghat}$, $W_{\gddot}$},
    tick label style={font=\Large},
]

\addplot[domain=-5:-2, samples=100, line width=1.5pt]{-0.3*x^2+8};

\addplot[domain=-2:6, samples=100, line width=1.5pt] {-0.2*(x+0.5)^2 + 6.5};

    \draw[dotted, line width=1.5pt] (-4,1) -- (-4,9);
    \draw[dotted, line width=1.5pt] (-2,1) -- (-2,9);

\end{axis}
\end{tikzpicture}
        \caption{$B > W_\gddot$}
        \label{fig: piecewise_3}
        \end{subfigure}
    \caption{Illustration for piecewise concave quadratic function \eqref{eq: rev_leader_marginal}}
    \label{fig: piecewise}
    \end{figure}
      
    Consider the value of $B$ corresponding to a set $\Ghat''\in\Gset$ . If $B \leq W_\ghat$, then $w''_1 q''_1$ monotonically decreases when $w''_1 > W_\ghat$, and thus $w''_1 q''_1$ cannot exceed the revenue when the leader bids a price no more than $W_\ghat$ and sells all its qualified capacity. Therefore, the leader does not deviate from its truthful bid when $B \leq W_\ghat$.

    Otherwise, if $B > W_\ghat$, the maximum of $w''_1 q''_1$ can be achieved either at $B$ (if $B < W_{\gddot''}$) or at $W_{\gddot''}$. If $w''_1 = B$ is the maximizer, then the corresponding revenue is $\frac{B^2}{A}$, otherwise, the revenue is $W_{\gddot''}\left(\frac{\Pi^\max - W_{\gddot''}}{A} - \sum_{i\in\Ghat''\sm\{1\}} F^U_i P^\max_i\right)$. Note that $W_{\gddot''}$ is not necessarily illustrated in Figure \ref{fig: piecewise}, as in general $W_{\gddot''}\geq W_{\gddot}$.

    Therefore, it is beneficial for the leader to bid untruthfully if $\exists \bar{\Gset}$ such that $\forall \Ghat''\in\bar{\Gset}$, the corresponding $B > W_{\ghat}$, and if the revenue from truthful bidding, $W_{\ghat} F^U_1 P^\max_1$ is less than either of the following:
    
    (i) $\max_{\Ghat''\in\bar{\Gset}, B < W_{\gddot''}}\frac{B^2}{A}$;
    
    (ii)  $\max_{\Ghat''\in\bar{\Gset}}W_{\gddot''}\left(\frac{\Pi^\max - W_{\gddot''}}{A} - \sum_{i\in\Ghat''\sm\{1\}} F^U_i P^\max_i\right)$. 

    \noindent Note that for condition (i), $\max_{\Ghat''\in\bar{\Gset}, B < W_{\gddot''}}\frac{B^2}{A} > W_{\ghat} F^U_1 P^\max_1$ is equivalent to $\max_{\Ghat''\in\bar{\Gset}, B < W_{\gddot''}}B>\sqrt{A W_{\ghat}F^U_1 P^\max_1}$.
\Halmos\endproof
\begin{repeattheorem}[Proposition \ref{th: cm_stack_marginal}.]
    If the leader $g=1$ is the marginal supplier when bidding truthfully, then it could increase the revenue by either bidding 0 or bidding a different price as the marginal supplier. The leader prefers to still be the marginal supplier if $\exists \bar{\Gset}$ such that $\forall \Ghat''\in\bar{\Gset}$, the corresponding $B > W_{\ghat'}$, and if either of the following is true:
    
    (i) $\max_{\Ghat''\in\bar{\Gset}, B < W_{\gddot''}}B>\sqrt{A W_{\ghat'}F^U_1 P^\max_1}$;
    
    (ii) $\max_{\Ghat''\in\bar{\Gset}} W_{\gddot''} \left(\frac{\Pi^\max-W_{\gddot''}}{A}-\sum_{i\in\Ghat''\sm\{1\}}F^U_i P^\max_i\right) > W_{\ghat'} F^U_1 P^\max_1$.

    The leader's strategic offer price $w''_1$ equals either $B$ or $W_{\gddot''}$ depending on which price leads to a higher revenue. The leader bids truthfully only if $w''_1 = W_1$.
\end{repeattheorem}
\proof{Proof. of Proposition \ref{th: cm_stack_marginal}}
    If the leader bids 0, then its revenue is $W_{\ghat'}F^U_1 P^\max_1$.  To find conditions under which it is beneficial for the leader to be the marginal supplier, we can use the results in Proposition \ref{th: cm_stack_allocated}, by treating 0 as the leader's truthful bidding price and $W_{\ghat'}$ as the corresponding market clearing price. Thus, the leader earns a higher revenue than $W_{\ghat'}F^U_1 P^\max_1$ if $\exists \bar{\Gset}$ such that $\forall \Ghat''\in\bar{\Gset}$, the corresponding $B > W_{\ghat'}$, and if either of the following conditions is true:
    
    (i) $\max_{\Ghat''\in\bar{\Gset}, B < W_{\gddot''}}B>\sqrt{A W_{\ghat'}F^U_1 P^\max_1}$;
    
    (ii) $\max_{\Ghat''\in\bar{\Gset}} W_{\gddot''} \left(\frac{\Pi^\max-W_{\gddot''}}{A}-\sum_{i\in\Ghat''\sm\{1\}}F^U_i P^\max_i\right) > W_{\ghat'} F^U_1 P^\max_1$.

    Case (i) corresponding to the leader bidding $w''_1 = B$ for a $\Ghat''\in\bar{\Gset}$ that maximizes the left-hand side of the inequality; Case (ii) corresponding to the leader bidding $w''_1 = W_\gddot''$ for a $\Ghat''\in\bar{\Gset}$ that maximizes the left-hand side of the inequality. If both cases are true, then the leader picks a bidding price (either $B$ or $W_{\gddot''}$) that leads to the higher revenue. This price may or may not be the same as its net CONE.

    If the conditions stated above are not satisfied, then the leader maximizes its revenue by bidding 0.
\Halmos\endproof
\begin{repeattheorem}[Proposition \ref{th: cm_stack_unallocated}.]
If the leader $g=1$ is in the set $\mc{G}\sm\Ghat$ when bidding truthfully, then it always earns a higher revenue when bidding untruthfully. More specifically, the leader maximizes its revenue either by bidding 0 or by becoming the marginal supplier. It becomes the marginal supplier if $F^U_1 P^\max_1 \geq \frac{\Pi^{\max}-W_{\ghat}}{A} - \sum_{i\in\Ghat\sm\{\ghat\}} F^U_i P^\max_i$ and $\exists \bar{\Gset}$ such that $\forall \Ghat''\in\bar{\Gset}$, the corresponding $B > W_{\ghat'}$, and if either of the following is true:

(i) $\max_{\Ghat''\in\bar{\Gset}, B < W_{\gddot''}}B>\sqrt{A W_{\ghat'}F^U_1 P^\max_1}$;

(ii) $\max_{\Ghat''\in\bar{\Gset}} W_{\gddot''} \left(\frac{\Pi^\max-W_{\gddot''}}{A}-\sum_{i\in\Ghat''\sm\{1\}}F^U_i P^\max_i\right) > W_{\ghat'} F^U_1 P^\max_1$.
\end{repeattheorem}
\proof{Proof. of Proposition \ref{th: cm_stack_unallocated}}
    Since $g=1\in\G\sm\Ghat$, when it bids 0, the new market clearing price $W_{\ghat'} \leq W_{\ghat}$. Also, when it becomes the marginal supplier, in order to get allocated its bidding price $w''_1 \leq W_{\ghat}$. 

    When the leader bids 0, its revenue is $W_{\ghat'}F^U_1 P^\max_1$. To earn a higher revenue as the marginal supplier, we need $w''_1 > W_{\ghat'} \Rightarrow w''_1\in(W_{\ghat'}, W_\ghat]$. Therefore, the leader will not be the marginal supplier if $W_{\ghat'} = W_{\ghat}$, which is true when $r^* = \frac{\Pi^{\max}-W_{\ghat}}{A} > \sum_{i\in\Ghat\sm\{\ghat\}} F^U_i P^\max_i + F^U_1 P^\max_1 \Rightarrow F^U_1 P^\max_1 < \frac{\Pi^{\max}-W_{\ghat}}{A} - \sum_{i\in\Ghat\sm\{\ghat\}} F^U_i P^\max_i$. 
    
    Otherwise, if $W_{\ghat'} < W_{\ghat}$ and thus $F^U_1 P^\max_1 \geq \frac{\Pi^{\max}-W_{\ghat}}{A} - \sum_{i\in\Ghat\sm\{\ghat\}} F^U_i P^\max_i$, then it is possible for the leader to obtain a higher revenue than $W_{\ghat'}F^U_1 P^\max_1$ by bidding $w''_1 \in (W_{\ghat'}, W_{\ghat}]$. To find conditions under which it is beneficial for the leader to become the marginal supplier, we again use the results in Proposition \ref{th: cm_stack_allocated}, similar to what we did in the proof of Proposition \ref{th: cm_stack_marginal}. By treating 0 as the leader's truthful bidding price and $W_{\ghat'}$ as the corresponding market clearing price, we can conclude that in this case the leader earns a higher revenue as the marginal supplier if $\exists \bar{\Gset}$ such that $\forall \Ghat''\in\bar{\Gset}$, the corresponding $B > W_{\ghat'}$, and if either of the following conditions is true:
    
    (i) $\max_{\Ghat''\in\Gset, B < W_{\gddot''}}B>\sqrt{A W_{\ghat'}F^U_1 P^\max_1}$;
    
    (ii) $\max_{\Ghat''\in\Gset} W_{\gddot''} \left(\frac{\Pi^\max-W_{\gddot''}}{A}-\sum_{i\in\Ghat''\sm\{1\}}F^U_i P^\max_i\right) > W_{\ghat'} F^U_1 P^\max_1$.
\Halmos\endproof
\begin{repeattheorem}[Proposition \ref{th: em_lmp}.]
If the network is connected and without congestion, then the value of $\lambda^*_{it}$ is the same at all buses for each time period. If there is physical withholding by the leader, then $\lambda^{*'}_{i(1), t}\geq \lambda^{*}_{i(1), t}, \forall t\in\mc{T}$.
\end{repeattheorem}
\proof{Proof. of Proposition \ref{th: em_lmp}}

    To prove $\lambda^*_{it}$ is the same at all buses, we use an alternative expression for the energy market price \citep{ji2016probabilistic} based on the shift factor formulation of (EM):
    \[
    \lambda^*_{it} = \hat{\lambda}^*_t + \sum_{(j, k)\in\mc{E}}B_{jk}(U_{ji}-U_{ki})(-\zeta^*_{jkt}+\eta^*_{jkt}),
    \]
    where $\hat{\lambda}_t$ is a dual variable of the shift factor formulation. $U_{ji}$ and $U_{ki}$ are parameters in the shift factor formulation. When there is no congestion in the connected network, constraints \eqref{eq: em_3} and \eqref{eq: em_4} are not binding. By complementary slackness, their corresponding optimal dual solutions $\zeta^*_{ijt} = \eta^*_{ijt} = 0$. Therefore, $\lambda^*_{it} = \hat{\lambda}^*_t, \forall i\in\mc{N}$. That is to say, $\lambda^*_{it}$ is the same at all buses for $t$.

    When there is no congestion, there is only one marginal generator in the network at time $t$, which we denote as $\gtilde$. Since we minimize total costs, generators with $C_g^V<C_{\gtilde}^V$ must operate at full capacity, while generators with $C_g^V>C_{\gtilde}^V$ do not produce. As shown in the proof of Proposition \ref{th: marg_gen}, when $p^*_{\gtilde t}<P^\max_{\gtilde}$ the price $\lambda^{*}_{i(\gtilde), t} =C_{\gtilde}^V$. Thus, $\lambda^{*}_{it} = C_{\gtilde}^V,\forall i\in\mc{N}$. 
    
    Let $\gddot$ be the least expensive generator that is not operating (we assume for now that $\gtilde$ is not the most expensive generator). We show that when $p^*_{\gtilde t} = P^{\max}_{\gtilde}$, $\lambda^{*}_{i(1), t}\in[C_{\gtilde}^V, C_{\gddot}^V]$. On one hand, as shown in the proof of Proposition \ref{th: marg_gen}, when the generator operates at full capacity $\lambda^{*}_{i(1), t}\geq C_{\gtilde}^V$. On the other hand, since the capacity constraint for $\gddot$ is not a strict equality, $\alpha^*_{\gddot t} = 0$ and thus $\lambda^{*}_{i(1), t}\leq C_{\gddot}^V$ because of dual constraint \eqref{eq: em_kkt_1}. 

    Now let $\gtilde'$ be the new marginal generator at $t$ when the leader withholds its capacity. Then $C_{\gtilde'}^V \geq C_{\gtilde}^V$. The updated LMP can be similarly calculated as before. In particular:

    (1) If $\gtilde' = \gtilde$, then $p^*_{\gtilde t}<p^{*'}_{\gtilde t}\leq P^{\max}_{\gtilde}$. Therefore, $\lambda^{*}_{i(1), t} = C_{\gtilde}^V$ and $\lambda^{*'}_{i(1), t} \in [C_{\gtilde}^V, C_{\gddot}^V]$. Thus, $\lambda^{*'}_{i(1), t} \geq \lambda^{*}_{i(1), t}$;

    (2) If $\gtilde' \neq \gtilde$, then $C_{\gtilde'}^V \geq C_{\gddot}^V$ and thus $\lambda^{*'}_{i(1), t} \in [C_{\gddot}^V, C_{\gddot'}^V]$. Therefore, $\lambda^{*'}_{i(1), t} \geq \lambda^{*}_{i(1), t}$.

    Finally, if $\gtilde$ is the most expensive generator, then after withholding the price either stays the same or increases to $\cvoll$.
\Halmos\endproof
\section{MIP Model for Capacity Market}\label{ch: mip}
In this section, we propose a MIP model for the capacity market. First, we sort the generators in the ascending order of $W_g$. Let $x_g$ be a binary variable that equals 1 if and only if generator $g$ gets allocated in the capacity market, and $z_g$ be a binary variable that equals 1 if and only if generator $g$ is the marginal producer. The MIP model is as follows:
\begin{subequations}\label{eq: cm_mip}
\begin{alignat}{4}
\text{(CM-MIP):}~\min~~&\sum_{g\in\mc{G}} W_g z_g \label{eq: cm_mip_0}\\
\st~~& h_g \leq F^U_g P^\max_g && g\in\mc{G}\label{eq: cm_mip_0_1}\\
& q_g \leq h_g && g\in\mc{G}\label{eq: cm_mip_0_2}\\
& q_g\leq F^U_g P^\max_g x_g && g\in\mc{G}\label{eq: cm_mip_1}\\
& q_g \geq F^U_g P^\max_g (x_g - z_g) && g\in\mc{G}\label{eq: cm_mip_3}\\
& \sum_{i = 1}^g q_{i} \geq \Qsold_g z_g && g\in\mc{G}\label{eq: cm_mip_4}\\
&  \sum_{i = 1}^g q_{i} \leq \Qsold_g z_g  + M (1 - z_g)~~&& g\in\mc{G}\label{eq: cm_mip_5}\\
& x_g = \sum_{i = g}^{|\mc{G}|}z_{i}~~&&\forall g\in\mc{G}\label{eq: cm_mip_6}\\
& \sum_{g\in\mc{G}}z_g = 1\label{eq: cm_mip_7}\\
& x_g, z_g\in\{0,1\}&&\forall g\in\mc{G}\label{eq: cm_mip_8}
\end{alignat}
\end{subequations}
where $\Qsold_g = \frac{\Pi^\max - W_g}{A}$ is the total sold capacity when market clearing price equals $W_g$, and $M$ is a big-M constant. The objective \eqref{eq: cm_mip_0} minimizes the market clearing price, which equals the net CONE of the marginal producer. Constraints \eqref{eq: cm_mip_0_1} set the upper bound for the offer capacity. Constraints \eqref{eq: cm_mip_0_2} set the upper bound for the sold capacity. Constraints \eqref{eq: cm_mip_1} connect variables $q_g$ and $x_g$, and force $q_g$ to be 0 when $x_g = 0$. Constraints \eqref{eq: cm_mip_3} ensure that if a non-marginal generator sells some of its capacity then all its capacity will be sold. Constraints \eqref{eq: cm_mip_4} and \eqref{eq: cm_mip_5}  ensure that if $g$ is a marginal generator, then the total sold capacity up to generator $g$ should equal $\Qsold_g$. Otherwise, those constraints are redundant. We can set the big-M constant $M = \Qsold_1$. This is because the total sold capacity is at most $\Qsold_1$, as the lowest possible market clearing price equals the lowest net CONE $W_1$. Constraints \eqref{eq: cm_mip_6} connect $x_g$ and $z_g$. Constraint \eqref{eq: cm_mip_7} ensures that there is only 1 marginal producer.

\newpage
\section{Tables for Case Study}\label{ec: tables}
\FloatBarrier

\begin{table}[htbp]
 \captionsetup[table]{justification=raggedright,singlelinecheck=off}
\caption{Comparison for the leader's profit (\$) from the energy market ($^*$: potentially underestimated difference due to suboptimality)}
\centering
\label{tab:diff_profit_1}
{\small
\begin{tabular}{lrrrlrr}
\hline
\multicolumn{1}{c}{}       & \multicolumn{1}{c}{}                  & \multicolumn{2}{c}{Normal congestion}                                               & \multicolumn{1}{c}{} & \multicolumn{2}{c}{High congestion}                                                \\ \cline{3-4} \cline{6-7} 
\multicolumn{1}{c}{Leader} & \multicolumn{1}{c}{Demand (\%)} & \multicolumn{1}{c}{S. Both vs. noCM} & \multicolumn{1}{c}{S. Both vs. True} & \multicolumn{1}{c}{} & \multicolumn{1}{c}{S. Both vs. noCM} & \multicolumn{1}{c}{S. Both vs. True} \\ \hline
NG                         & 60                                    & 0.0                                   & 0.0                                    &                      & 0.0                                   & 0.0                                    \\
NG                         & 70                                    & 0.0                                   & 0.0                                    &                      & 0.0                                   & 0.0                                    \\
NG                         & 80                                    & 0.0                                   & 0.0                                    &                      & 0.0                                   & 0.0                                    \\
NG                         & 90                                    & -355.1                                & 0.0                                    &                      & -718.9                                & 0.0                                    \\
NG                         & 100                                   & -263.1                                & 0.0                                    &                      & -29110.7                              & 14501.2                                \\
NG                         & 110                                   & 0.0                                   & 0.0                                    &                      & 0.0                                   & 0.0                                    \\
NG                         & 120                                   & 0.0                                   & 0.0                                    &                      & 0.0                                   & 0.0                                    \\
Coal                       & 60                                    & -1242.7                               & 566.9                                  &                      & 0.0                                   & 0.0                                    \\
Coal                       & 70                                    & 0.0                                   & 0.0                                    &                      & 0.0                                   & 0.0                                    \\
Coal                       & 80                                    & -539.8                                & 0.0                                    &                      & -386.5                                & 0.0                                    \\
Coal                       & 90                                    & -173.2                                & 0.0                                    &                      & 0.0                                   & 0.0                                    \\
Coal                       & 100                                   & -182.0                                & 0.0                                    &                      & -9032.2                               & 14760.0                                \\
Coal                       & 110                                   & 0.0                                   & 0.0                                    &                      & 0.0                                   & 0.0                                    \\
Coal                       & 120                                   & 0.0                                   & 0.0                                    &                      & 0.0                                   & 0.0                                    \\
Nuclear                    & 60                                    & 0.0                                   & 650.3                                  &                      & 0.0                                   & 0.0                                    \\
Nuclear                    & 70                                    & 0.0                                   & 0.0                                    &                      & 0.0                                   & 0.0                                    \\
Nuclear                    & 80                                    & 0.0                                   & 994.0                                  &                      & 0.0                                   & 691.2                                  \\
Nuclear                    & 90                                    & 0.0                                   & 0.0                                    &                      & 0.0                                   & 0.0                                    \\
Nuclear                    & 100                                   & -443.3                                & 0.0                                    &                      & -28330.0                              & 29289.1                                \\
Nuclear                    & 110                                   & 0.0                                   & 0.0                                    &                      & 0.0                                   & 0.0                                    \\
Nuclear                    & 120                                   & 0.0                                   & 0.0                                    &                      & -1403738.2                            & 0.0                                    \\
RFO                        & 60                                    & 0.0                                   & 0.0                                    &                      & 0.0                                   & 0.0                                    \\
RFO                        & 70                                    & 0.0                                   & 0.0                                    &                      & 0.0                                   & 0.0                                    \\
RFO                        & 80                                    & 0.0                                   & 0.0                                    &                      & 0.0                                   & 0.0                                    \\
RFO                        & 90                                    & 0.0                                   & 0.0                                    &                      & 0.0                                   & 0.0                                    \\
RFO                        & 100                                   & 0.0                                   & 0.0                                    &                      & 0.0                                   & 0.0                                    \\
RFO                        & 110                                   & 0.0                                   & 0.0                                    &                      & 0.0                                   & 0.0                                    \\
RFO                        & 120                                   & 0.0                                   & 0.0                                    &                      & -668428.0                             & 0.0                                    \\
Hydro                      & 60                                    & -809.5                                & 0.0                                    &                      & $0.0^*$           & 0.0                                    \\
Hydro                      & 70                                    & 0.0                                   & 0.0                                    &                      & $-121.9^*$        & 0.0                                    \\
Hydro                      & 80                                    & -166.1                                & 0.0                                    &                      & -2388.5                               & 0.0                                    \\
Hydro                      & 90                                    & 0.0                                   & 0.0                                    &                      & $0.0^*$           & 0.0                                    \\
Hydro                      & 100                                   & 0.0                                   & 0.0                                    &                      & -4565.8                               & 0.0                                    \\
Hydro                      & 110                                   & 0.0                                   & 0.0                                    &                      & 0.0                                   & 0.0                                    \\
Hydro                      & 120                                   & 0.0                                   & 0.0                                    &                      & 0.0                                   & 0.0                                    \\
Wood                       & 60                                    & 0.0                                   & 0.0                                    &                      & 0.0                                   & 0.0                                    \\
Wood                       & 70                                    & 0.0                                   & 0.0                                    &                      & 0.0                                   & 0.0                                    \\
Wood                       & 80                                    & 0.0                                   & 0.0                                    &                      & 0.0                                   & 0.0                                    \\
Wood                       & 90                                    & 0.0                                   & 0.0                                    &                      & 0.0                                   & 0.0                                    \\
Wood                       & 100                                   & 0.0                                   & 0.0                                    &                      & -201.2                                & 0.0                                    \\
Wood                       & 110                                   & 0.0                                   & 0.0                                    &                      & 0.0                                   & 0.0                                    \\
Wood                       & 120                                   & 0.0                                   & 0.0                                    &                      & 0.0                                   & 0.0                                    \\ \hline
\end{tabular}}
\end{table}
\begin{table}[htbp]
\centering
\caption{Comparison for the leader's profit (\$) from the energy market, allowing untruthful bids for the leader's variable cost ($^*$: potentially underestimated difference due to suboptimality)}
\label{tab:diff_profit_2}
{\small
\begin{tabular}{lrrrlrr}
\hline
\multicolumn{1}{c}{}       & \multicolumn{1}{c}{}                  & \multicolumn{2}{c}{Normal congestion}                                        & \multicolumn{1}{c}{} & \multicolumn{2}{c}{High congestion}                                         \\ \cline{3-4} \cline{6-7} 
\multicolumn{1}{c}{Leader} & \multicolumn{1}{c}{Demand (\%)} & \multicolumn{1}{c}{S. Both vs. noCM} & \multicolumn{1}{c}{S. Both vs. True} & \multicolumn{1}{c}{} & \multicolumn{1}{c}{S. Both vs. noCM} & \multicolumn{1}{c}{S. Both vs. True} \\ \hline
NG                         & 60                                    & 0.0                               & 0.0                                 &                      & 0.0                               & 0.0                                 \\
NG                         & 70                                    & 0.0                               & 0.0                                 &                      & 0.0                               & 0.0                                 \\
NG                         & 80                                    & 0.0                               & 0.0                                 &                      & 0.0                               & 0.0                                 \\
NG                         & 90                                    & -96.4                             & 564.7                               &                      & -29.3                             & 771.2                               \\
NG                         & 100                                   & 0.0                               & 18849.3                             &                      & 0.0                               & 64715.8                             \\
NG                         & 110                                   & 0.0                               & 61.1                                &                      & $0.0^*$       & 269.3                               \\
NG                         & 120                                   & 0.0                               & 118488.0                            &                      & 0.0                               & 5471.8                              \\
Coal                       & 60                                    & $-1242.7^*$   & 566.9                               &                      & 0.0                               & 0.0                                 \\
Coal                       & 70                                    & 0.0                               & 0.0                                 &                      & 0.0                               & 248.5                               \\
Coal                       & 80                                    & -539.8                            & 0.0                                 &                      & -386.5                            & 0.0                                 \\
Coal                       & 90                                    & -173.2                            & 0.0                                 &                      & 0.0                               & 0.0                                 \\
Coal                       & 100                                   & $-182.0^*$    & 0.0                                 &                      & -7920.3                           & 16035.9                             \\
Coal                       & 110                                   & 0.0                               & 0.0                                 &                      & 0.0                               & 0.0                                 \\
Coal                       & 120                                   & 0.0                               & 61801.5                             &                      & 0.0                               & 0.0                                 \\
Nuclear                    & 60                                    & 0.0                               & 650.3                               &                      & 0.0                               & 0.0                                 \\
Nuclear                    & 70                                    & 0.0                               & 0.0                                 &                      & 0.0                               & 0.0                                 \\
Nuclear                    & 80                                    & 0.0                               & 994.0                               &                      & 0.0                               & 691.2                               \\
Nuclear                    & 90                                    & 0.0                               & 0.0                                 &                      & 0.0                               & 0.0                                 \\
Nuclear                    & 100                                   & -443.3                            & 0.0                                 &                      & -25972.7                          & 31813.7                             \\
Nuclear                    & 110                                   & 0.0                               & 0.0                                 &                      & -0.1                              & 0.0                                 \\
Nuclear                    & 120                                   & 0.0                               & 884096.4                            &                      & 0.0                               & 2137730.4                           \\
RFO                        & 60                                    & 0.0                               & 0.0                                 &                      & 0.0                               & 0.0                                 \\
RFO                        & 70                                    & 0.0                               & 0.0                                 &                      & 0.0                               & 0.0                                 \\
RFO                        & 80                                    & 0.0                               & 0.0                                 &                      & 0.0                               & 0.0                                 \\
RFO                        & 90                                    & 0.0                               & 0.0                                 &                      & 0.0                               & 0.0                                 \\
RFO                        & 100                                   & 0.0                               & 0.0                                 &                      & 0.0                               & 0.0                                 \\
RFO                        & 110                                   & 0.0                               & 0.0                                 &                      & 0.0                               & 0.0                                 \\
RFO                        & 120                                   & 0.0                               & 731223.9                            &                      & 0.0                               & 936940.8                            \\
Hydro                      & 60                                    & 0.0                               & 1411.6                              &                      & $0.0^*$       & 0.0                                 \\
Hydro                      & 70                                    & 0.0                               & 591.2                               &                      & 0.0                               & 181.4                               \\
Hydro                      & 80                                    & 0.0                               & 543.7                               &                      & 0.0                               & 5106.5                              \\
Hydro                      & 90                                    & 0.0                               & 649.8                               &                      & 0.0                               & 2225.0                              \\
Hydro                      & 100                                   & 0.0                               & 0.0                                 &                      & 0.0                               & 10137.5                             \\
Hydro                      & 110                                   & 0.0                               & 0.0                                 &                      & 0.0                               & 259301.9                            \\
Hydro                      & 120                                   & 0.0                               & 0.0                                 &                      & 0.0                               & 164498.9                            \\
Wood                       & 60                                    & 0.0                               & 0.0                                 &                      & 0.0                               & 0.0                                 \\
Wood                       & 70                                    & 0.0                               & 0.0                                 &                      & 0.0                               & 0.0                                 \\
Wood                       & 80                                    & 0.0                               & 0.0                                 &                      & 0.0                               & 0.0                                 \\
Wood                       & 90                                    & 0.0                               & 0.0                                 &                      & 0.0                               & 0.0                                 \\
Wood                       & 100                                   & 0.0                               & 0.0                                 &                      & 0.0                               & 868.6                               \\
Wood                       & 110                                   & 0.0                               & 0.0                                 &                      & 0.0                               & 987.6                               \\
Wood                       & 120                                   & 0.0                               & 0.0                                 &                      & 0.0                               & 1240.4                              \\ \hline
\end{tabular}}
\end{table}

\begin{table}[htbp]
\centering
\caption{Computation time (seconds) of high congestion noCM instances without and with the valid inequality ({\bf bold numbers}: improved run time with valid inequality; ``tl.": hits the 20-minute time limit)}
\label{tab: ineq_time}
{\small
\begin{tabular}{lrrr}
\hline
\multicolumn{1}{c}{Leader} & \multicolumn{1}{c}{Demand (\%)} & \multicolumn{1}{c}{No ineq.} & \multicolumn{1}{c}{With ineq.} \\ \hline
NG                         & 60                              & 4.5                          & 5.9                            \\
NG                         & 70                              & 4.3                          & 4.6                            \\
NG                         & 80                              & 4.3                          & 5.0                            \\
NG                         & 90                              & 4.2                          & 5.5                            \\
NG                         & 100                             & 5.4                          & 7.2                            \\
NG                         & 110                             & 6.5                          & 11.6                           \\
NG                         & 120                             & 7.5                          & 9.1                            \\
Coal                       & 60                              & 15.0                         & 17.5                           \\
Coal                       & 70                              & 6.7                          & 8.6                            \\
Coal                       & 80                              & 5.9                          & \textbf{5.8}                   \\
Coal                       & 90                              & 5.3                          & 6.1                            \\
Coal                       & 100                             & 14.3                         & 17.3                           \\
Coal                       & 110                             & 11.5                         & 78.7                           \\
Coal                       & 120                             & 9.7                          & 11.8                           \\
Nuclear                    & 60                              & 18.6                         & 19.1                           \\
Nuclear                    & 70                              & 10.2                         & 13.5                           \\
Nuclear                    & 80                              & 7.7                          & 10.4                           \\
Nuclear                    & 90                              & 6.2                          & 7.6                            \\
Nuclear                    & 100                             & 17.0                         & 27.4                           \\
Nuclear                    & 110                             & 10.2                         & 18.2                           \\
Nuclear                    & 120                             & 20.7                         & \textbf{18.1}                  \\
RFO                        & 60                              & 3.8                          & 8.0                            \\
RFO                        & 70                              & 5.0                          & \textbf{4.3}                   \\
RFO                        & 80                              & 5.3                          & \textbf{4.8}                   \\
RFO                        & 90                              & 5.3                          & 5.3                  \\
RFO                        & 100                             & 5.2                          & \textbf{5.1}                   \\
RFO                        & 110                             & 4.6                          & 4.8                            \\
RFO                        & 120                             & 5.1                          & \textbf{5.0}                   \\
Hydro                      & 60                              & tl.                     & tl.                       \\
Hydro                      & 70                              & tl.                     & tl.                       \\
Hydro                      & 80                              & tl.                     & \textbf{44.6}                  \\
Hydro                      & 90                              & tl.                     & tl.                       \\
Hydro                      & 100                             & 13.9                         & tl.                       \\
Hydro                      & 110                             & 5.2                          & 6.8                            \\
Hydro                      & 120                             & 5.2                          & 6.7                            \\
Wood                       & 60                              & 4.4                          & 7.8                            \\
Wood                       & 70                              & 5.5                          & 7.3                            \\
Wood                       & 80                              & 4.5                          & 6.9                            \\
Wood                       & 90                              & 4.4                          & 5.5                            \\
Wood                       & 100                             & 4.7                          & 6.1                            \\
Wood                       & 110                             & 6.4                          & \textbf{6.0}                   \\
Wood                       & 120                             & 5.1                          & 5.6                            \\ \hline
\end{tabular}}
\end{table}

%
%
%






\end{document}